\documentclass[11pt]{article}

\usepackage{amsmath,amsthm,amssymb}
\usepackage{graphicx}
\usepackage{appendix}

\font\tencmmib=cmmib10 \skewchar\tencmmib '60
\newfam\cmmibfam
\textfont\cmmibfam=\tencmmib

\def\lessim{\ \lower4pt\hbox{$
		\buildrel{\displaystyle <}\over\sim$}\ }
\def\gessim{\ \lower4pt\hbox{$\buildrel{\displaystyle >}
		\over\sim$}\ }

\def\Bla{\Big{\langle}}
\def\Bra{\Big{\rangle}}
\def\bla{\langle}
\def\bra{\rangle}
\def\la{\langle}
\def\ra{\rangle}

\newcommand{\e}{\mathbb{E}}
\newcommand{\p}{\mathbb{P}}

\newcommand{\indi}{\ensuremath{\boldsymbol 1}}
\newcommand{\Crt}{\mathop{\mathrm{Crt}}\nolimits}

\newtheorem{lemma}{\bf Lemma}

\newtheorem{theorem}{\bf Theorem}

\newtheorem{remark}{\bf Remark}

\newtheorem{proposition}{\bf Proposition}

\newenvironment{Proof of lemma}{\noindent{\bf Proof of Lemma}}{\hfill$\Box$\newline}
\newenvironment{Proof of theorem}{\noindent{\bf Proof of Theorem}}{\hfill{\footnotesize${\square}$}\newline}
\newenvironment{Proof of theorems}{\noindent{\bf Proof of Theorems}}{\hfill$\Box$\newline}
\newenvironment{Proof of proposition}{\noindent{\bf Proof of Proposition}}{\hfill$\Box$\newline}
\newenvironment{Proof of propositions}{\noindent{\bf Proof of Propositions}}{\hfill$\Box$\newline}
\newenvironment{Proof of exercise}{\noindent{\it Proof of Exercise:}}{\hfill$\Box$}

\begin{document}
	
	\nocite{*} 
	
	\title{Parisi formula, disorder chaos and fluctuation for\\ 
		the ground state energy in the spherical mixed $p$-spin models}
	
	\author{Wei-Kuo Chen\thanks{School of Mathematics, University of Minnesota. Email: wkchen@umn.edu.}
		\and
		Arnab Sen\thanks{School of Mathematics, University of Minnesota. Email: arnab@umn.edu}
	}
	\maketitle
	
	\begin{abstract}
		
		We show that the limiting ground state energy of the spherical mixed $p$-spin model can be identified as the infimum of certain variational problem. This complements the well-known Parisi formula for the limiting free energy in the spherical model. As an application, we obtain explicit formulas for the limiting ground state energy in the replica symmetry, one level of replica symmetry breaking and full replica symmetry breaking phases at zero temperature. In addition, our approach leads to new results on disorder chaos in spherical  mixed  even $p$-spin models. In particular, we prove that when there is no external field, the location of the ground state energy is chaotic under small perturbations of the disorder. We also establish that in the spherical mixed even $p$-spin model, the ground state energy superconcentrates in the absence of external field, while it obeys a central limit theorem if the external field is present. 
	\end{abstract}
	
	\tableofcontents
	
	\section{Introduction}
	For $N\geq 1,$ let $S_N=\{\sigma\in\mathbb{R}^N:\sum_{i=1}^N\sigma_i^2=N\}$ be the sphere of radius $\sqrt{N}$. The Hamiltonian of the spherical mixed $p$-spin model is defined as
	\begin{align*}
	H_N(\sigma)&=X_N(\sigma)+h\sum_{i=1}^N\sigma_i,   \quad \forall\sigma \in S_N,
	\end{align*} 
	where $X_N$ is the following centered Gaussian process indexed by $S_N$,
	\begin{align}\label{def:X_N}
	X_N(\sigma)&=\sum_{p\geq 2}\frac{\gamma_p}{N^{(p-1)/2}}\sum_{1\leq i_1,\ldots,i_{p}\leq N}g_{i_1,\ldots,i_{p}}\sigma_{i_1}\cdots\sigma_{i_{p}}
	\end{align}
	for i.i.d.\ standard Gaussian random variables $g_{i_1,\ldots,i_p}$ for $1\leq i_1,\ldots,i_{p}\leq N$ and $p\geq 2$.  Here, $h$ denotes the strength of the external field and the sequence $(\gamma_p)$ stands for the mixture parameter that is assumed to decay fast enough, for instance, $\sum_{p\geq 2}2^p\gamma_p^2<\infty,$ such that the infinite sum in \eqref{def:X_N} converges a.s.  
    This allows us to compute
	$$
	\e X_N(\sigma^1)X_N(\sigma^2)=N\xi(R(\sigma^1,\sigma^2)),
	$$
	where $$\xi(s):=\sum_{p\geq 2}\gamma_p^2s^{p}$$ and $$
	R(\sigma^1,\sigma^2):=\frac{1}{N}\sum_{i=1}^N\sigma_i^1\sigma_i^2$$ 
	is the overlap between any two spin configurations $\sigma^1$ and $\sigma^2.$ An important case of $\xi$ is the spherical mixed even $p$-spin model, i.e., $\gamma_p=0$ for all odd $p\geq 3.$ The spherical Sherrington-Kirkpatrick (SK) model corresponds to $\xi(s)=s^2/2$. In contrast, the {\it Ising} mixed $p$-spin model is defined essentially in the same way only now the configuration space $S_N$ is replaced by $\Sigma_N=\{-1,+1\}^N.$
	
	The aim of this article is to investigate the behavior of the ground state energy
	$$L_N:=\max_{\sigma\in S_N}H_N(\sigma)$$
	and the ground state
	$$
	\sigma^*:=\mbox{argmax}_{\sigma\in S_N}H_N(\sigma)
	$$
	in the thermodynamic limit $N\rightarrow\infty$. Our main results consist of three major parts. First, we establish a variational formula for the limit of the scaled ground state energy. This limit should be seen as the analogue of the Parisi formula (see \eqref{CS:eq1} below) at zero temperature. As a consequence, we extend the fundamental concepts of replica symmetry and replica symmetry breaking for the limiting free energy to the ground state energy and we obtain explicit and simple expressions for this limit in the replica symmetry, one level of replica symmetry breaking and full replica symmetry breaking phases. Second, we prove disorder chaos for the ground state in any spherical mixed even $p$-spin models regardless of the presence or absence of the external field. Third, we show that the variance of $L_N$ is sublinear for any spherical mixed even $p$-spin models in the absence of external field, $h=0.$ This establishes superconcentration of the ground state energy. In the case that the external field is present, $h\neq 0,$ we further obtain a central limit theorem for the ground state energy. These results will be discussed in detail in the following three subsections.

	\subsection{The Parisi formula for the ground state energy}\label{sub1.1}
	
	For any $\beta>0$, define the free energy  by
	$$
	F_N(\beta)=\frac{1}{N}\log Z_N(\beta),
	$$
	where $Z_N(\beta)$ is the partition function defined as
	\begin{align}
	\label{more:eq5}
	Z_N(\beta)=\int_{S_N}\exp\bigl( \beta H_N(\sigma)\bigr)m_N(d\sigma)
	\end{align} for $m_N$ the uniform probability measure on $S_N$. Here, the parameter $\beta$ is called the {\it inverse} temperature. We say that the model is at {\it positive} temperature if $\beta<\infty$ and is at {\it zero} temperature if $\beta=\infty.$ Set 
	\begin{align*}
	\xi_\beta=\beta^2\xi\,\,\mbox{and}\,\,h_\beta=\beta h.
	\end{align*} 
	Let $\mathcal{M}$ be the space of all distribution functions $x$ on $[0,1]$ with $x(\hat{q})=1$ for some $\hat{q}<1.$ For any $x\in\mathcal{M}$ and $b\in\mathbb{R}$ satisfying
	\begin{align}\label{eq3}
	b>\max\Big\{1,\int_0^1\xi_\beta''(s)x(s)ds\Big\},
	\end{align}
	define the Parisi functional by
	\begin{align}\label{parisi}
	\mathcal{P}_\beta(x,b)=\frac{1}{2}\Bigl(\frac{h_\beta^2}{b-d_\beta^x(0)}+\int_0^1\frac{\xi_\beta''(q)}{b-d_\beta^x(q)}dq+b-1-\log b-\int_0^1q\xi_\beta''(q)x(q)dq\Bigr),
	\end{align}
	where $d_\beta^x(q):=\int_q^1\xi_\beta''(s)x(s)ds.$ The famous Parisi formula for the limiting free energy of the spherical  mixed-$p$ spin model says that
	\begin{align}\label{prop3:proof:eq0}
	F(\beta):=\lim_{N\rightarrow\infty}\e F_N(\beta)&=\inf \mathcal{P}_\beta(x,b),
	\end{align}
	where the infimum is over all pairs $(x,b)$ that satisfy \eqref{eq3} and is uniquely achieved by some pair $(x_{\beta},b_{\beta})$. 
	This formula was initially proved in the case of the mixed even $p$-spin models by Talagrand \cite{Tal06}. The general situation was handled by Chen \cite{Chen13}.
	Throughout this paper, we call the probability measure $\mu_\beta$ induced by $x_\beta$ the Parisi measure. In physics literature (see e.g. Mezard-Parisi-Virasoro \cite{MPV}), the system is called replica symmetric if $\mu_\beta$ is a Dirac measure, $k$ replica symmetry breaking if it is an atomic measure with exactly $k+1$ atoms and full replica symmetry breaking otherwise. We refer the readers to Talagrand \cite{Tal06} and Auffinger-Chen \cite{AChen13} for some examples of Parisi measures.
	
	 Alternatively, the Parisi formula admits a simpler expression, discovered by Crisanti-Sommers \cite{CS}. For any $x\in\mathcal{M}$, set
	$$
	\hat{x}(q)=\int_q^1x(s)ds,\,\,\forall q\in[0,1].
	$$
	Define the Crisanti-Sommers functional by
	\begin{align}\label{csf}
	\mathcal{Q}_\beta(x)&=\frac{1}{2}\Bigl(\int_0^1(\xi_\beta'(q)+h_\beta^2)x(q)dq+\int_{0}^{\hat{q}}\frac{dq}{\hat{x}(q)}+\log(1-\hat{q})\Bigr),
	\end{align}
	where $\hat{q}\in[0,1)$ satisfies $x(\hat{q})=1$. Note that this functional is well-defined as $\mathcal{Q}_\beta(x)$ is independent of the choice of $\hat{q}.$ The Parisi formula can also be written as
		\begin{align}
		\label{CS:eq1}
		F(\beta)&=\inf_{x\in\mathcal{M}}\mathcal{Q}_\beta(x).
		\end{align}
		Here the minimizer is uniquely attained by the minimizer $x_\beta$ of \eqref{prop3:proof:eq0} (see \cite{Tal06}).

		It is well-known that the limiting ground state energy exists and can be computed through  the limiting free energy (see Lemma \ref{lem:discrete} below), as
			\begin{align}\label{eq:EquationGSdefinition}
			GS:=\lim_{N\rightarrow\infty}\frac{ L_N}{N}=\lim_{\beta\rightarrow\infty}\frac{F(\beta)}{\beta},
			\end{align}
			where the first limit above exists in $L^1$ and a.s. However, it is far from clear that one can obtain a meaningful expression for the limit on the right-hand side of \eqref{eq:EquationGSdefinition} from the intricate variational problems  \eqref{prop3:proof:eq0} and \eqref{CS:eq1}.  This is the content of our first main result. We obtain an analogue of the Parisi formula for the limiting ground state energy via the Crisanti-Sommers representation \eqref{CS:eq1}. It is described as follows.
		
		 Denote by $\mathcal{N}$ the collection of all nonnegative nondecreasing and right-continuous functions on $[0,1)$.
		 Let 
		 \begin{align}
		 \label{more:eq10}
		 		 \mathcal{K}:=\Big\{(L,\alpha)\in (0,\infty)\times\mathcal{N}:L>\int_0^1\alpha(s)ds\Bigr\}.
		 \end{align}
		 For any $(L,\alpha)\in\mathcal{K}$, define 
	$$
	\mathcal{Q}(L,\alpha):=\frac{1}{2}\Bigl((\xi'(1)+h^2)L-\int_0^1\xi''(q)\Bigl(\int_0^q\alpha(s)ds\Bigr)dq+\int_0^1\frac{dq}{L-\int_0^q\alpha(s)ds}\Bigr).
	$$
	Note that $\mathcal{Q}$ defines a strictly convex functional on the convex space $\mathcal{K}$, but $\mathcal{K}$ is not compact. Below is our main result.
	
	\begin{theorem}[Parisi formula for the ground state energy]
		\label{add:thm1}
		We have that
		\begin{align}
		\label{add:thm1:eq1}
		GS&=\min_{(L,\alpha)\in\mathcal{K}}\mathcal{Q}(L,\alpha),
		\end{align}
		where the minimizer is unique and is given by the pair $(L_0,\alpha_0)\in\mathcal{K}$ for 
		\begin{align}
		\begin{split}\label{add:eq6}
		L_0&:=\lim_{\beta\rightarrow\infty}\int_0^1\beta x_\beta(s)ds,\\
		\alpha_0&:=\lim_{\beta\rightarrow\infty}\beta x_\beta\,\,\mbox{vaguely on $[0,1)$}.
		\end{split}
		\end{align} 		
	\end{theorem}
	
	Here the existence of the last two limits is part of the main result. 
	
	\begin{remark} \label{rm1}\rm  The vague convergence of $(\beta x_\beta)_{\beta>0}$ on $[0,1)$ in \eqref{add:eq6} means that $\lim_{\beta\rightarrow\infty}\beta x_\beta(s)=\alpha_0(s)$ at all points of continuity of $\alpha_0$ on $[0,1)$. It is equivalent to the statement that $$\lim_{\beta\rightarrow\infty}\int_0^1f(s)d(\beta x_\beta)=\int_0^1f(s)d\alpha_0$$ for all continuous functions $f$ on $[0,1)$ with compact support. Likewise, for any nonnegative sequence $(\beta_n)_{n\geq 1}$ with $\lim_{n\rightarrow\infty}\beta_n=\infty$, we define the vague convergence of $(\beta_nx_{\beta_n})_{n\geq 1}$ on $[0,1)$ in the same fashion.
	\end{remark} 
	
	\begin{remark}
		\rm Recall that $GS$ depends on the mixture parameter $(\gamma_p)_{p\geq 2}$ and the external field $h$.  In \cite{Chen16}, it was known that from the Parisi formula \eqref{add:thm1:eq1}, $GS$ is partially differentiable in each $\gamma_p$ and $h$,
		\begin{align*}
		\partial_{h}GS&=hL_0,\\
		\partial_{\gamma_p}GS&=p\gamma_p\Bigl(L_0-\int_0^1\alpha_0(s)ds+\int_0^1\alpha_0(s)s^{p-1}ds\Bigr).
		\end{align*} 
		As a consequence, the magnetization and the pure $p$-spin Hamiltonians of the ground state $\sigma^*$ converge, 
		\begin{align*}
		&\p\Bigl(\Bigl|\frac{\sum_{i=1}^N\sigma_i^*}{N}-\partial_hGS\Bigr|\geq \varepsilon\Bigr)\leq Ke^{-N/K},\\
		&\p\Bigl(\Bigl|\frac{H_{N,p}(\sigma^*)}{N}-\partial_{\gamma_p}GS\Bigr|\geq \varepsilon\Bigr)\leq Ke^{-N/K}
		\end{align*} 
		for all $N\geq 1$, where $K$ is a positive constant independent of $N.$ Note that the first inequality was established in \cite{Chen16}. One may adapt exactly the same argument to derive the second one.
	\end{remark}

	Theorem \ref{add:thm1} suggests one way of constructing the minimizer of \eqref{add:thm1:eq1}, but it remains very difficult to compute $(L_0,\alpha_0)$ as one needs the precise expression of the Parisi measures at any positive temperature. Using the strict convexity of $\mathcal{Q}$, we establish a direct characterization for the minimizer of \eqref{add:thm1:eq1} that avoids taking $\beta\rightarrow\infty.$

	\begin{theorem}\label{add:thm:char} 
		Let $(L,\alpha)\in\mathcal{K}$. Define 
		\begin{align}\label{add:thm:char:eq3}
		g(u)&=\int_u^1\bar{g}(s)ds,\,\,\forall u\in[0,1],
		\end{align}
		where
		\begin{align}\label{add:thm:char:eq2}
		\bar{g}(s):=\xi'(s)+h^2-\int_{0}^s\frac{dq}{\bigl(L-\int_0^q\alpha(r)dr\bigr)^2},\, \,\,\forall s\in[0,1].
		\end{align}
		Then $(L,\alpha)$ is the minimizer of \eqref{add:thm1:eq1} if and only if the following equation holds,
		\begin{align}\label{add:thm:char:eq1}
		\xi'(1)+h^2=\int_0^1\frac{dq}{\bigl(L-\int_0^q\alpha(r)dr\bigr)^2}
		\end{align}
		and the function $g$ satisfies $\min_{u\in [0,1]}g(u)\geq 0$ and $\nu(S)=\nu([0,1))$, where $\nu$ is the measure induced by $\alpha$, i.e., $\nu([0, s]) = \alpha(s)$ for all $ s\in [0,1)$ and $$S:=\{u\in[0,1):g(u)=0\}.$$
	\end{theorem}
	
	\begin{remark}
		\rm
		Our approach of Theorems \ref{add:thm1} and \ref{add:thm:char} adapts the Crisanti-Sommers expression \eqref{CS:eq1}. Following a similar argument presented in this paper, the results analogous to Theorems 1 and 2 might be obtained  by utilizing the Parisi formula \eqref{prop3:proof:eq0}. Nonetheless, the corresponding variational representation for the maximal energy will involve an additional variable $b$ similar to the one appearing in the functional $\mathcal{P}_\beta(x,b).$
	\end{remark}

		Theorem \ref{add:thm:char} allows us to extend the regions of replica symmetry and replica symmetric breaking for Parisi measures at {\it positive} temperature to {\it zero} temperature. Let $(L_0,\alpha_0)$ be the minimizer of \eqref{add:thm1:eq1} and $\nu_0$ be the measure induced by $\alpha_0.$ Analogous to Parisi's formulation, we say that the model {\it at zero temperature} is replica symmetric if $\nu_0=0$, $k$ levels of replica symmetry breaking if $\nu_0$ is an atomic measure with $k$ atoms and full replica symmetry breaking otherwise. The proposition below characterizes the region of replica symmetric minimizers.
		%

	
	\begin{proposition}
		\label{add:prop1}
	The model is replica symmetric at zero temperature if and only if 
	\begin{align}
	\label{add:prop1:Eq1}
	\xi'(1)+h^2\geq\xi''(1).
	\end{align} In this case, the minimizer $(L_0,\alpha_0)$ of \eqref{add:thm1:eq1} is given by $L_0=(\xi'(1)+h^2)^{-1/2}$ and $\alpha_0=0$. Furthermore,
			\begin{align}\label{GS:eq1}
			GS&=(\xi'(1)+h^2)^{1/2}.
			\end{align}
	\end{proposition}

		The assumption $\xi'(1)+h^2\geq\xi''(1)$ in Proposition \ref{add:prop1} is equivalent to  
		\begin{align*}
				h^2\geq\sum_{p\geq 3}(p^2-2p)\gamma_p^2.
		\end{align*}
		The equality here plays the role of the famous de Almeida-Thouless line \cite{AT} at zero temperature. 
		If \eqref{add:prop1:Eq1} is violated and $\xi''(s)^{-1/2}$ is concave on $(0,1]$, we further obtain:

	\begin{proposition}
		\label{add:prop1.5}
			If $\xi'(1)+h^2<\xi''(1)$ and $\xi''(s)^{-1/2}$ is concave on $(0,1]$, then the model is full replica symmetry breaking at zero temperature. In this case, the minimizer $(L_0,\alpha_0)$ of \eqref{add:thm1:eq1} is given by $L_0=\xi''(q_0)^{-1/2}$ and 
			\begin{align*}
			\alpha_0(s)&=\left\{
			\begin{array}{ll}
			0,&\mbox{if $q\in[0,q_0)$},\\
			\frac{\xi'''(s)}{2\xi''(s)^{3/2}},&\mbox{if $q\in[q_0,1)$}
			\end{array}
			\right.
			\end{align*}
			and the ground state energy is equal to
			\begin{align}\label{GS:eq2}
			GS&=q_0\xi''(q_0)^{1/2}+\int_{q_0}^1 \xi''(q)^{1/2}dq,
			\end{align}
			where the quantity $q_0\in[0,1]$ is the unique solution to 
			\begin{align}
			\label{GS:eq3}
			\xi'(q_0)+h^2=q_0\xi''(q_0).
			\end{align}
	\end{proposition}
	
	Consider the spherical pure $p$-spin model for $p\geq 3$ in the absence of external field, i.e., $\xi(s)=s^p/p$ and $h=0.$ It is easy to see that \eqref{add:prop1:Eq1} does not hold and $\xi''(s)^{-1/2}$ is convex on $(0,1].$ In this case, we show that the model is $1$-replica symmetry breaking.
	
	\begin{proposition}\label{add:prop2}
		Assume that $\xi(s)=s^p/p$ for $p\geq 3$ and $h=0.$ Then the model is $1$-replica symmetry breaking at zero temperature. Here the minimizer $(L_0,\alpha_0)$ of \eqref{add:thm1:eq1} is  given by $L_0=(z\delta)^{1/2}+(z^{-1}\delta)^{1/2}$ and $\alpha_0=(z\delta)^{1/2}$ for $\delta:=z(1+z)^{-1}$ and
		\begin{align}
		\label{add:prop2:eq1}
		GS&=\frac{1}{\sqrt{z+1}}\Bigl(1+\frac{z}{p}\Bigr),
		\end{align}
		where $z>0$ is the unique solution to 
	\begin{align}\label{add:prop2:eq2}
	\frac{1}{p}&=\frac{1+z}{z^2}\log (1+z)-\frac{1}{z}.
	\end{align}
	\end{proposition}
	
	A word of comment is needed here. The ground state energy for the spherical pure $p$-spin model without external field was previously studied in Auffinger-Ben Arous-\v{C}ern\'{y} \cite{ACB} through the complexity of the local minima of the Hamiltonian. Equation \eqref{add:prop2:eq1} matches (with different normalization) the constant $E_0(p)$ defined in \cite[Theorem 2.12]{ACB}. When $p\geq 7$, the correct order of the fluctuation of the ground state energy was obtained in Subag-Zeitouni \cite{SZ}. One of the crucial ingredients in their work is the concentration of the number of critical points of $H_N$ that was conjectured in Auffinger-Ben Arous \cite{ABA13} and verified in Subag \cite{S}. When combined with the results of \cite{ACB}, Proposition  \ref{add:prop1.5} implies that such concentration phenomenon does not hold for all spherical mixed $p$-spin models. More precisely, for any $u\in \mathbb R$, define the
	(random) number
	$\Crt_{N,0}(u)$ of local minima of the
	function $H_{N}$ below the level $Nu$ as
	%
	\begin{equation*}
	\label{defWk} \Crt_{N,0}(u) = \sum_{\sigma: \nabla
		H_{N}(\sigma) = 0 }
	\indi\bigl\{ H_{N}(\sigma) \leq Nu \bigr\} \indi\bigl\{ i
	\bigl(\nabla^2 H_{N}(\sigma)\bigr) = 0\bigr\}.
	\end{equation*}
	Here $\nabla$, $\nabla^2$ are the gradient and the Hessian restricted to
	$S_{N}$, and $i(\nabla^2 H_{N}(\sigma))$ is the
	number of negative eigenvalues
	of the Hessian $\nabla^2 H_{N}$, called the index of the Hessian at
	$\sigma$.
	In \cite{ACB, ABA13}, the following limit was established for any spherical mixed  $p$-spin model: 
	$$  \lim_{N \to \infty}\frac{1}{N} \log \mathbb E \Crt_{N,0}(u) = \Theta_0(u),$$ 
	where $\Theta_0(u)$, called the complexity of local minima, is an increasing function defined on an interval of the form $(-\infty, -E_\infty)$ for some $E_\infty >0$. This function has a unique zero denoted by $-E_0$ and it only depends on the model through the values of $\xi(1)$, $\xi'(1)$ and $\xi''(1)$. It was observed in \cite{ABA13} that if for some $\epsilon >0$ and $u \in(-E_0, -E_0+\epsilon)$, 
	\begin{equation}\label{eq:concentrationforcomplexity}
	 \lim_{N \to \infty}\frac{\Crt_{N,0}(u)}{\mathbb E \Crt_{N,0}(u)} = 1 \quad \text{a.s.},
	\end{equation}
  then 
	$
	E_0 = GS.
	$
	Proposition \ref{add:prop1.5} now implies that \eqref{eq:concentrationforcomplexity} does not hold if  $\xi'(1)<\xi''(1)$ and $\xi''(s)^{-1/2}$ is concave on $(0,1]$,  since the ground state energy is equal to 
		\begin{align*}
		GS&=\int_{0}^1 \xi''(q)^{1/2}dq
		\end{align*}
	 and this quantity is not always determined by just $\xi(1)$, $\xi'(1)$ and $\xi''(1)$. This conclusion illustrates that the complexity of certain spherical mixed $p$-spin models does not concentrate. It contrasts to the common assumption of self-averaging in several physics papers, see e.g. Crisanti-Leuzzi-Rizzo \cite{CLR}, Crisanti-Leuzzi \cite{CL2004} and Kurchan-Parisi-Virasoro\cite{KPV}.
	%
	%
	
	\smallskip
	\smallskip

		{\noindent \bf Open Problem.} {\rm The proof of Theorem \ref{add:thm1} relies heavily on the Parisi formula at positive temperature established in \cite{Chen13,Tal06}. It would be of great interest to see whether the methodologies of \cite{Chen13,Tal06} can be used to give a direct proof of Theorem~\ref{add:thm1}.}

	\subsection{Chaos in disorder for the ground state}
	
	Chaos in disorder is concerned with the	phenomenon that in some spin glass models, a small perturbation to the disorder will result in a drastic change to the overall energy landscape.  Over the past decades, this subject has received a lot of attention in physics community, see Bray-Moore \cite{BM87}, Fisher-Huse \cite{FH86}, Kr\c{z}aka\l{}a-Bouchaud \cite{KB05} for physics literature and Rizzo \cite{R2009} for an up-to-date survey. Recently, several mathematical results on chaos in disorder for the overlap at positive temperature are also made available: Chatterjee \cite{Chatt13} obtained disorder chaos for the Ising mixed even $p$-spin models  without external field and Chen \cite{Chen12} carried out the situation when the external field is present and extended the results to some Ising mixed $p$-spin models allowing odd $p$-spin interactions, see \cite{Chen150}. More recently, chaos in disorder is also obtained in the spherical mixed even $p$-spin model by Chen-Hsieh-Hwang-Sheu \cite{Chen151}. 
	
	Our result here establishes chaos in disorder for the ground state overlap at zero temperature. Assume that the Gaussian part $X_N$ of the Hamiltonian $H_N$ is even, i.e., $\gamma_p=0$ for all odd $p\geq 3.$ Similar to the formulation of the research works mentioned above, we consider two i.i.d. copies $X_N^1$ and $X_N^2$ of $X_N.$ Set two spherical mixed even $p$-spin Hamiltonians,
	\begin{align}
	\begin{split}
	\label{hamilton}
	H_{N,t}^1(\sigma)&=\sqrt{t}X_N(\sigma)+\sqrt{1-t}X_N^1(\sigma)+h\sum_{i=1}^N\sigma_i,\\
	H_{N,t}^2(\tau)&=\sqrt{t}X_N(\tau)+\sqrt{1-t}X_N^2(\tau)+h\sum_{i=1}^N\tau_i,
	\end{split}
	\end{align}
	where $t\in[0,1]$ is called the coupling parameter. In other words, the two systems differ by the independent Gaussian Hamiltonians $X_N^1$ and $X_N^2$.
	Let $\sigma_t^*$ and $\tau_t^*$ be the ground states of $H_{N,t}^1(\sigma)$ and $H_{N,t}^2(\tau)$ over $S_N$, respectively, i.e., 
	\begin{align}\label{more:eq4}
	\sigma_t^*=\mbox{argmax}_{\sigma\in S_N}H_N^1(\sigma)\,\,\mbox{and}\,\,\tau_t^*=\mbox{argmax}_{\tau\in S_N}H_N^2(\tau).
	\end{align} 
    Note that $X_{N},X_{N,t}^1,X_{N,t}^2$ are linear combinations of independent even $p$-spin interactions. If $h=0,$ then each $H_{N,t}^j$ for $j=1,2$ has two optimizers and they are different by a minus sign, while for $h\neq 0,$ each $H_{N,t}^j$ has a unique maximizer. If $t=1$, then the two systems are identically the same and the overlap has the relation that $|R(\sigma_t^*,\tau_t^*)|=1$ if $h=0$ and $R(\sigma_t^*,\tau_t^*)=1$ if $h\neq 0.$

	If now the two systems are decoupled, $0<t<1$, we show that this behavior will change drastically in such a way that the ground state overlap $R(\sigma_t^*,\tau_t^*)$  is essentially concentrated around zero if the external field is absent. In the presence of external field, we prove that this overlap is concentrated near a constant $u_t\in (0,q_0)$ for some $q_0\in[0,1].$ More precisely, let $(L_0,\alpha_0)$ be the minimizer of \eqref{add:thm1:eq1}. Define $q_0=\min\mbox{supp}\alpha_0$ if $\mbox{supp}\alpha_0\neq \emptyset$ and $q_0=1$ if $\mbox{supp}\alpha_0=\emptyset.$ Our main result is stated in the following theorem.

	\begin{theorem}[Chaos in disorder for ground states]\label{thm4}
		Consider the spherical mixed even $p$-spin coupled Hamiltonian \eqref{hamilton}.
		For any $0<t<1,$  there exists some $u_t\in[0,1)$ such that for any $\varepsilon\in(0,1),$
		\begin{align*}
		\lim_{N\rightarrow\infty}\p(|R(\sigma_t^*,\tau_t^*)-u_t|\geq \varepsilon)=0,
		\end{align*}	
		where $u_t$ satisfies the equation 
		$$
		L_0^2\bigl(t\xi'(u_t)+h^2\bigr)=u_t
	    $$ and it has the property that $u_t=0$ if $h=0$ and $u_t\in(0,q_0)$ if $h\neq 0.$
	\end{theorem}

		The above theorem says that the ground state overlap $R(\sigma_t^*,\tau_t^*)$ is nearly a constant value $u_t$, or equivalently, the distance between the ground states $\sigma_t^*$ and $\tau_t^*$ is around $\sqrt{2(1-u_t)}$ under the normalized Euclidean distance. In the absence of external field, this distance is essentially $\sqrt{2}$, while if the external field is present, it is at least $\sqrt{2(1-q_0)}$, where this lower bound is positive for some examples of the spherical mixed $p$-spin models as illustrated by Proposition \ref{add:prop1.5}. In contrast to the situation that $\sigma_t^*=\tau_t^*$ when $t=1$, these illustrate that, as long as the two systems are decoupled, $t\in (0,1)$, we immediately witness the strict separation of the ground states by a positive distance independent of the choice of $0<t<1.$ This confirms the chaotic nature of the spherical mixed even $p$-spin model under perturbations of the disorder. We remark that in the special case of the spherical SK model with no external field, the maximizers $\sigma_t^*$ and $\tau_t^*$ actually correspond to the first eigenvectors of two correlated $N\times N$ GOE matrices in distribution, in which case an upper bound for the second moment of the ground state overlap that deduces chaos in disorder was obtained by Chatterjee \cite{Chatt13}. 
		
		\subsection{Fluctuation of the ground state energy}

		Finally we  discuss the fluctuation of the ground state energy in the spherical mixed $p$-spin models. In the case of the spherical SK model without external field, i.e., $\xi(s)=s^2/2$ and $h=0$, the ground state energy has a simple alternative description, ${L_N}/N = \lambda_{\max}/2$ in distribution, where $\lambda_{\max}$ is  the maximum eigenvalue of an $N \times N$ GOE matrix. The classical result in random matrix theory says that the normalized $L_N$ converges weakly to the GOE Tracy-Widom Law. Recently, for the spherical pure $p$-spin model  with $p\geq 7$ and $h=0,$ Subag and Zeitouni \cite{SZ} showed that the centered $L_N$ converges to a Gumbel distribution. However, beyond these cases, the result on the fluctuation of the ground state energy for the general spherical mixed $p$-spin model is relatively scarce. Our first main result here says that for any spherical mixed even $p$-spin model with no external field, the ground state energy $L_N$  superconcentrates, which means that the variance of $L_N$ is of a smaller order than the one obtained through the Poincar\'{e} inequality, $\mathrm{Var}(L_N)= O(N)$.

		\begin{theorem}[Superconcentration]\label{thm2}
			For any mixed even $p$-spin model, if there is no external field, then
			\begin{align}
			\label{thm2:eq1}
			\lim_{N\rightarrow\infty}\frac{1}{N}\text{Var}(L_N)=0.
			\end{align}	
		\end{theorem}
		
		The limit \eqref{thm2:eq1} means that the variance of $L_N$ is of a sublinear order. It would be particularly interesting to derive a tight upper bound for $\mbox{Var}(L_N).$ The result along this direction was previously studied in the case of the Ising mixed $p$-spin model by Chatterjee \cite{Chatt13}, where he showed that $\mbox{Var}(L_N)=O(N^{3/4}(\log N)^{1/4})$ for certain choice of the mixture parameter $(\gamma_p)$ in the absence of external field. Probably the same approach would work in the spherical model too. We do not pursue this direction here. 
		
		In recent years, there have been some progress in understanding the fluctuation properties of the free energy. Chatterjee \cite{Chatt13} showed that the free energy of the Ising SK model without external field superconcentrates at any positive temperature and his techniques can be applied to show superconcentration in a more general class of Ising mixed $p$-spin models without external field. For the spherical SK model without external field, Baik and Lee \cite{Baik15} computed the correct order of the fluctuation of the free energy. In particular, they showed that the variance of the free energy at low temperature is of order $N^{2/3}$. This behavior changes dramatically in the presence of external field. In fact, it was showed in a very recent paper \cite{cpp} of Chen-Dey-Panchenko  that the true order of the fluctuation for the free energy at any positive temperature in both Ising and spherical mixed $p$-spin models with external field matches with that suggested by the Poincar\'{e} inequality. Furthermore, in the case of the mixed even $p$-spin models, they proved that the normalized fluctuation obeys a central limit theorem. Our second result below establishes an analogue of such limit theorem for the ground state energy.

		\begin{theorem}[Central limit theorem]\label{thm3}
			Consider the spherical mixed even $p$-spin model in the presence of external field. Recalling the quantity $u_t$ from Proposition \ref{lem6}, define
			$$
			\chi=\int_0^1\xi(u_t)dt>0.
			$$
            Then we have
			\begin{align*}
			\lim_{N\rightarrow\infty}d_{\text{TV}}\Bigl(\frac{L_N-\e L_N}{\sqrt{\chi N}},g\Bigr)=0,
			\end{align*}
			where for any two random variables $X$ and $Y,$
			 $$
			d_{\text{TV}}(X,Y):=\sup_{A}\bigl|\p(X\in A)-\p(Y\in A)\bigr|
			$$ 
			is the total variation distance between $X,Y$ and $g$ is a standard Gaussian random variable. 
		\end{theorem}
		
			The assumption that the spherical model is even in Theorems \ref{thm4}, \ref{thm2} and \ref{thm3} is for technical purposes. It would be of great interest to extend these results to more general spherical mixed $p$-spin models.

	\subsection{Structure of the paper}

	Section~\ref{sec:parisi} will establish the Parisi formula for the ground state energy utilizing the inverse temperature limit of the Crisanti-Sommers functional $\mathcal{Q}_\beta(x_\beta).$ From this, we derive the characterization of the minimizer of $\mathcal{Q}$ followed by the verifications of Propositions~\ref{add:prop1}, \ref{add:prop1.5} and \ref{add:prop2}. In Section~\ref{sec:dis}, we will introduce a two-dimensional Guerra-Talagrand replica symmetry breaking bound for the free energy of the coupled Hamiltonian $H_{N,t}^1(\sigma)+H_{N,t}^2(\tau)$ with overlap constraint. This inequality is a two-dimensional generalization of $\mathcal{P}_\beta$ and was previously derived in Chen-Hsieh-Hwang-Sheu \cite{Chen151} to investigate chaos in disorder at any positive temperature. From this bound, we adapt the Parisi functional $\mathcal{P}_\beta$ to control the ground state energy of $H_{N,t}^1(\sigma)+H_{N,t}^2(\tau)$ over all possible values of the overlaps and deduce disorder chaos in Theorem~\ref{thm4}. In Section $4$, we prove the superconcentration in Theorem~\ref{thm2} by means of Chatterjee's integral representation for the variance of the ground state energy, see \cite{Chatt13}. In addition, we prove the central limit theorem in Theorem \ref{thm3} via Stein's method. At the end of this paper, we gather a few technical lemmas in the appendix. They will be devoted to studying some regularity properties of the Hamiltonians $H_N,H_{N,t}^1,H_{N,t}^2$ in order to justify the validity of our control of the ground state energy through the free energy at any positive temperature.
    
    \smallskip
	\smallskip
	
	{\noindent \bf Acknowledgements.} The authors thank Dmitry Panchenko for some discussions at the early stage of this research work. Special thanks are due to Antonio Auffinger for illustrating the absence of concentration of $\Crt_{N,0}$ from Proposition \ref{add:prop1.5} to us and several helpful comments. Both authors are indebted to an anonymous referee for several helpful suggestions regarding the presentation of the paper. The research of A. S. is partly supported by NSF grant DMS-1406247. The research of W.-K. C. is partly supported by NSF grant DMS-1513605 and Hong Kong Research Grants Council GRF-14302515.

 	\section{Proof of the Parisi formula at zero temperature}\label{sec:parisi}
 	
 		This section is devoted to establishing the main results in Subsection \ref{sub1.1}. The approach of Theorem~\ref{add:thm1} is based on a subtle control of the scaled Crisanti-Sommers functional, $\beta^{-1}\mathcal{Q}_\beta(x_\beta)$, as the inverse temperature tends to infinity. The difficulties we will encounter are mainly due to the fact that we generally do not know the analytic behavior of $x_\beta(s)$ when $s$ is very close to $1$ and $\beta$ is sufficiently large. This makes it very hard to handle the limits of the integrals in $\beta^{-1}\mathcal{Q}_\beta(x_\beta)$ directly. In order to obtain a meaningful limit, we shall construct a subsequence of $(\beta x_\beta)$ that converges vaguely on $[0,1)$ in the sense of Remark \ref{rm1}.  The novelty of our approach is to rewrite the integrals therein in an elementary way (see the expression \eqref{add:eq0}). In the limit, this allows us to avoid dealing with the singularity at $1$ by introducing a new variable $L$ in $\mathcal{Q}$ and leads to the desired representation. The same idea will be employed repeatedly throughout the rest of the paper.
 		
 		In Subsection~\ref{subs2.1}, we establish some key lemmas that help to control the behavior of $(\beta x_\beta)$. They will play an essential role later in the proof of disorder chaos. In Subsection \ref{subs2.2}, we establish the proof of Theorems \ref{add:thm1} and \ref{add:thm:char}. The establishment of Propositions \ref{add:prop1}, \ref{add:prop1.5}, and \ref{add:prop2} will be presented in Subsection \ref{subs2.3}. 
 		
 	\subsection{Some auxiliary lemmas}\label{subs2.1}

 	 Recall the optimizer $(x_\beta,b_\beta)$ from  \eqref{prop3:proof:eq0}. Let $q_\beta$ be the smallest value of $q$ such that $x_\beta(q)=1.$ The first step of our approach is to construct a subsequence of $(\beta x_\beta)$ that has a weak limit. We begin with two technical lemmas that control the function $\beta x_\beta$ and some of its integrals.
 	 
 	 \begin{lemma}\label{add:lem}
 	 There exists a constant $C_{\xi}$ depending only on $\xi$ such that for any $\beta>0$, 
 	 \begin{align}\label{eq-1}
 	 	\beta x_{\beta}(q)\leq \frac{C_\xi}{\xi(1)-\xi(q)},\,\,\forall q\in[0,1).
 	 	\end{align}
 	 \end{lemma}
 	 
 	 \begin{proof}
 	 	Note that for any $N\geq 1,$
 	 	\begin{align}\label{addition:eq1}
 	 	\e\max_{\sigma\in S_N}\frac{X_N(\sigma)}{N}\leq C_\xi.
 	 	\end{align}
 	 	Here the constant $C_\xi>0$ depends only on $\xi$ and this inequality is obtained by using the Dudley entropy integral (see e.g. \cite[Equation $(1.5)$]{Tal14}). For the detailed derivation, see Remark \ref{rmk2} in the appendix. From Gaussian integration by parts, one has the following identity,
 	 	\begin{align}\label{sec3.1:eq1}
 	 	\beta\Bigl(\xi(1)-\e\bla \xi(R(\sigma^1,\sigma^2)\bra_\beta\Bigr)=\e \Bla\frac{X_N(\sigma)}{N}\Bra_\beta,
 	 	\end{align}
 	 	where $\bla\cdot\bra_\beta$ is the Gibbs average with respect to the Gibbs measure $G_{N,\beta}(\sigma)$ defined by 
 	 	\begin{align*}
 	 	G_{N,\beta}(\sigma)=\frac{\exp \beta H_N(\sigma)}{Z_N(\beta)}.
 	 	\end{align*}
 	 		It is well-known (see \cite{Tal06}) that the Parisi formula is differentiable in $\beta$, which yields
 	 		\begin{align*}
 	 		\lim_{N\rightarrow\infty}\e\bla \xi(R(\sigma^1,\sigma^2))\bra_\beta&=\int_0^1\xi(s)x_\beta(ds).
 	 		\end{align*}
 	 	Using this equation together with \eqref{addition:eq1} and \eqref{sec3.1:eq1} leads to
 	 	$$
 	 	\beta\Bigl(\xi(1)-\int_0^1 \xi(s)x_\beta(ds)\Bigr)\leq C_\xi,
 	 	$$
 	 	where this inequality used the trivial bound $\la X_N(\sigma)\ra_\beta\leq \max_{\sigma\in S_N}X_N(\sigma).$
 	  	Finally, applying integration by part to this equation gives 
 	 	\begin{align}\label{add:eq-1}
 	 	\int_0^1\xi'(s)\beta x_\beta(s)ds=\beta\Bigl(\xi(1)-\int_0^1\xi(s)x_\beta(ds)\Bigr)\leq C_\xi.
 	 	\end{align}
 	 	Now, since clearly $$
 	 	\int_q^1\xi'(s)\beta x_{\beta}(s)ds\geq \beta x_\beta(q)(\xi(1)-\xi(q)),\,\,\forall q\in [0,1],$$
 	 	the inequality \eqref{eq-1} follows by \eqref{add:eq-1}.
 	 \end{proof}
 	 
 	 \begin{lemma}
 	 	\label{add:lem2}
 	 	     There exists a constant $C_\xi'>0$ depending only on $\xi$ such that
 	 	     \begin{align}\label{add:lem:eq0}
 	 	     \limsup_{\beta\rightarrow\infty}\beta(1-q_\beta)\leq C_\xi'
 	 	     \end{align}
 	 	     and for any $\beta>0$,
 	 		 	\begin{align}
 	 		 \begin{split}\label{add:lem:eq1}
 	 		 \int_0^1\beta x_\beta(s)ds&\leq C_\xi',\\
 	 		  \int_0^1\xi''(s)\beta x_\beta(s)ds&\leq C_\xi',\\
 	 		 \int_0^1s\xi''(s)\beta x_\beta(s)ds&\leq C_\xi',
 	 		 \end{split}
 	 		 \end{align}	
 	 		 
 	 \end{lemma}
 
 	 \begin{proof} 
 	 	Applying $q=q_\beta$ to \eqref{eq-1} gives
 	 	\begin{align}\label{add:lem:proof:eq1}
 	 	\beta=\beta x_\beta(q_\beta)\leq \frac{C_\xi}{\xi(1)-\xi(q_\beta)}.
 	 	\end{align}
 	 	Since the left-hand side tends to infinity as $\beta\rightarrow\infty$, this inequality forces $\lim_{\beta\rightarrow\infty}q_\beta=1.$ On the other hand, from \eqref{add:lem:proof:eq1}, the mean value theorem and noting that $\xi$ is nondecreasing,
 	 	\begin{align*}
 	 	\beta\xi'(q_\beta)(1-q_\beta) \leq \beta (\xi(1)-\xi(q_\beta))\leq C_\xi.
 	 	\end{align*}
 	 	Consequently,
 	 	\begin{align*}
 	 	\limsup_{\beta\rightarrow \infty}\beta(1-q_\beta)\leq \frac{C_\xi}{\xi'(1)}.
 	 	\end{align*}
 	    Next, using  \eqref{eq-1} and \eqref{add:eq-1} gives
 	 \begin{align*}
 	 \int_0^1\beta x_\beta(s)ds&=\int_0^{1/2}\beta x_\beta(s)ds+\int_{1/2}^1\beta x_\beta(s)ds\\
 	 &\leq \int_0^{1/2}\frac{C_\xi}{\xi(1)-\xi(q)}dq+\int_{1/2}^1\frac{\xi'(s)}{\xi'(1/2)}\beta x_\beta(s)ds\\
 	 &\leq \frac{C_\xi}{\xi(1)-\xi(1/2)}+\frac{C_\xi}{\xi'(1/2)}=:C.
 	 \end{align*}
 	 Note that since $\xi$ is nondecreasing,
 	 \begin{align*}
 	 \int_0^1\xi''(s)\beta x_\beta(s)ds&\leq \xi''(1)C,\\
 	 \int_0^1s\xi''(s)\beta x_\beta(s)ds&\leq \xi''(1) C.
 	 \end{align*}
 	 Letting $C_\xi':=\max(C_\xi/\xi'(1),C,\xi''(1)C)$, the above four inequalities give \eqref{add:lem:eq0} and the three inequalities in \eqref{add:lem:eq1}.
 	 \end{proof}
 	  
 	 From \eqref{eq-1}, we may use the Helly selection theorem combined with a diagonalization process to conclude that there exists a nonnegative and nondecreasing sequence $(\beta_n)_{n\geq 1}$ with $\lim_{n\rightarrow\infty}\beta_n=\infty$ such that $(\beta_nx_{\beta_n})_{n\geq 1}$ converges vaguely  on $[0,1)$ (in the sense of Remark \ref{rm1}). Furthermore, from \eqref{add:lem:eq0} and \eqref{add:lem:eq1}, we can pass to a subsequence $(\beta_{n_k})_{k\geq 1}$ of $(\beta_n)_{n\geq 1}$ (two times if necessary) such that along this common subsequence, the following limits exist,
 	 \begin{align*}
 	 &\lim_{k\rightarrow\infty}\beta_{n_k}x_{n_k}\in\mathcal{N}\,\,\mbox{vaguely on $[0,1)$},\\
 	 &\lim_{k\rightarrow\infty}\beta_{n_k}(1-q_{\beta_{n_k}}),\\
 	 &\lim_{k\rightarrow\infty}\int_0^1\beta_{n_k} x_{\beta_{n_k}}(s)ds,
 	 \end{align*}
 	 where the space $\mathcal{N}$ is defined right before \eqref{more:eq10}.
 	 To lighten the notation, we shall assume, without loss of generality, that all these convergences hold for the sequence $(\beta_n)_{n\geq 1}.$ Denote
 	 \begin{align}
 	 	 \begin{split}
 	 	 \label{eq-5}
 	 	 \alpha_0&:=\lim_{n\rightarrow\infty}\beta_nx_{\beta_n}\in\mathcal{N}\,\,\mbox{vaguely on $[0,1)$},\\
 	 \delta_0&:=\lim_{n\rightarrow\infty}\beta_n(1-q_{\beta_n}),\\
 	 	L_0&:=\lim_{n\rightarrow\infty}\int_0^1\beta_n x_{\beta_n}(s)ds.
 	 \end{split}
 	 \end{align} 
 	 Note that $\delta_0,L_0\leq C_\xi'$ by Lemma \ref{add:lem2}.
 	
 	The following lemma gathers some crucial properties of the quantities $q_\beta$ and $b_\beta$. They are deduced using the minimality of $(x_\beta,b_\beta)$ of the Parisi functional $\mathcal{P}_\beta$ in \eqref{prop3:proof:eq0}. Recall that $\mu_\beta$ denotes the probability measure induced by $x_\beta.$
 	
 	\begin{lemma}
 		\label{lem-2}
 		For any $q$ in the support of $\mu_\beta,$ we have that
 		\begin{align}
 		\begin{split}\label{add:thm1:proof:eq2}
 		q&=\frac{h_\beta^2}{(b_\beta-d_\beta^{x_\beta}(0))^2}+\int_{0}^{q}\frac{\xi_\beta''(s)}{(b_\beta-d_\beta^{x_\beta}(s))^2}ds
 		\end{split}
 		\end{align}
 		and
 		\begin{align}\label{add:thm1:proof:eq-1}
 		b_\beta-d_\beta^{x_\beta}(q)&=\frac{1}{\int_q^1x_\beta(s)ds}.
 		\end{align}		
 	\end{lemma}
 	
 	\begin{proof}
 	Let $x\in\mathcal{M}$. Denote $x_\theta=(1-\theta)x+\theta x_{\beta}$ for $\theta\in[0,1].$ Plugging $x_\theta$ into $\mathcal{P}_\beta(\cdot,b_\beta)$ and computing the derivative with respect to $\theta$, the minimality of $x_\beta$ and Fubini's theorem together yields
 	\begin{align}\label{lem-2:proof:eq1}
 	\partial_\theta \mathcal{P}_\beta(x_\theta,b_\beta)\big|_{\theta=0}&=\frac{1}{2}\int_0^1\Gamma(q)\xi_\beta''(q)(x(q)-x_\beta(q))dq\geq 0,
 	\end{align}
 	where $$
 	\Gamma(q):=\frac{h_\beta^2}{\bigl(b_\beta-d_\beta^{x_\beta}(0)\bigr)^2}+\int_0^q\frac{\xi_\beta''(s)ds}{\bigl(b_\beta-d_\beta^{x_\beta}(s)\bigr)^2}-q
 	$$
 	Let $\mu$ be the probability measure induced by $x$. Writing $x(q)-x_\beta(q)=\int_0^qd(\mu-\mu_\beta)(r)$ and using Fubini's theorem again, \eqref{lem-2:proof:eq1} can be translated into
 	\begin{align*}
 	\int_0^1\int_r^1\Gamma(q)\xi_\beta''(q)dqd(\mu-\mu_\beta)(r)\geq 0
 	\end{align*}
 	and thus,
 	\begin{align*}
 	\int_0^1\bar{\Gamma}(r)d\mu(r)\geq \int_0^1\bar{\Gamma}(r)d\mu_\beta(r)
 	\end{align*}
 	for $\bar{\Gamma}(r):=\int_r^1\Gamma(q)\xi_\beta''(q)dq.$ Since the last inequality holds for all $x,$ this is equivalent to say that 
 	$$
 	\bar{\Gamma}(s)\geq \int_0^1\bar{\Gamma}(r)d\mu_\beta(r)
 	$$
 	for all $s\in [0,1]$ and the equality holds for every point in $\mbox{supp}(\mu_\beta)$. Note that $\mbox{supp}(\mu_\beta)\subset[0,1).$ If $s\in\mbox{supp}(\mu_\beta)\cap(0,1)$, we have
 	$$
 	\frac{d}{ds}\bar{\Gamma}(s)=-\Gamma(s)\xi_\beta''(s)=0
 	$$ 
    and then $\Gamma(s)=0$ since $\xi_\beta''>0$ on $(0,1)$.  If $s=0\in \mbox{supp}(\mu_\beta),$ the mean value theorem implies that for any $\eta\in(0,1)$, $$
 	-\Gamma(\eta')\xi_\beta''(\eta')=\bar{\Gamma}(\eta)-\bar{\Gamma}(0)\geq 0
 	$$
 	for some $\eta'\in(0,\eta),$ which leads to $\Gamma(0)\leq 0.$ This could be true only if $h_\beta=0$, in which case evidently $\Gamma(0)=0.$ These lead to \eqref{add:thm1:proof:eq2}. 
 	
 	As for \eqref{add:thm1:proof:eq-1}, for $k\geq 1,$ denote by $\mathcal{M}_k$ the space of all step functions $x\in\mathcal{M}$ with at most $k$ jumps and by $\mathcal{M}_k'$ the space of all $(x,b)$ with $x\in\mathcal{M}_k$ and $b\in\mathbb{R}$ satisfying $$b>\int_0^1\xi_\beta''(s)\beta x(s)ds.$$
 	Let $(x_k,b_k)$ be the minimizer of $\mathcal{P}_\beta$ restricted to $\mathcal{M}_k'.$ From \cite[Section 4]{Tal06}, it is understood that $x_k$ is also the minimizer of $\mathcal{Q}_\beta$ restricted to $\mathcal{M}_k.$ Furthermore, according to \cite[Equation $(4.11)$]{Tal06}, $(x_k,b_k)$ satisfies the following equation,
 	\begin{align}
 	\label{eq-11}
 	b_k-d_\beta^{x_k}(q)=\frac{1}{\int_q^1x_k(s)ds}
  	\end{align}
 	for all $q$ in the support of the probability measure $\mu_k$ induced by $x_k.$ By the uniqueness of the minimizer $(x_\beta,b_\beta)$ of $\mathcal{P}_\beta$, we may pass to a subsequence of $(x_k,b_k)$ such that its limit equals $(x_\beta,b_\beta).$ Thus, \eqref{add:thm1:proof:eq-1} follows from \eqref{eq-11}.	
 	\end{proof}
 	     
 	     The above proof is the only place where we need the Parisi formula \eqref{prop3:proof:eq0} in this section. It will be heavily used again when we establish disorder chaos in Section \ref{sec:dis}.
 	     
 	Letting $q=q_\beta$, the equation \eqref{add:thm1:proof:eq-1} reads
 	\begin{align}
 	\begin{split}
 	\label{add:thm1:proof:eq0}
 	b_\beta&=\xi_\beta'(1)-\xi_\beta'(q_\beta)+\frac{1}{1-q_\beta}.
 	\end{split}
 	\end{align} 
 	Lemma \ref{lem-2} allows us to prove some key properties of $\delta_0,L_0,\alpha_0.$
 	
 	\begin{lemma}
 		\label{add:lem0}
 	$0<\delta_0<\infty$ and $(L_0,\alpha_0)\in\mathcal{K}$.
 	\end{lemma}
 	
 	\begin{proof}
 		The finiteness of $\delta_0$ has been verified by \eqref{add:lem:eq0}.
 		To show $\delta_0>0$, we argue by contradiction. Assume on the contrary $\delta_0=0.$ From \eqref{add:thm1:proof:eq0}, 
 		\begin{align*}
 		\lim_{n\rightarrow\infty}\frac{b_{\beta_n}}{\beta_n}=\xi''(1)\delta_0+\frac{1}{\delta_0}=\infty.
 		\end{align*}
 		On the other hand, from the second inequality of \eqref{add:lem:eq1}, $\beta_n^{-1}d_{\beta_n}^{x_{\beta_n}}(q)\leq C_\xi'$ for any $q\in[0,1]$. This and the above display together imply that 
 		\begin{align*}
        \frac{1}{\beta_n^{-1}\bigl(b_{\beta_n} -d_{\beta_n}^{x_{\beta_n}}(q)\bigr)}
 		\end{align*}
 		is uniformly bounded on $[0,1]$ for $n$ sufficiently large and it converges to zero uniformly on $[0,1].$ As a result, from \eqref{add:thm1:proof:eq2} with $q=q_{\beta_n}$, we reach a contradiction,
 		\begin{align*}
 		0&=\lim_{n\rightarrow\infty}\Bigl(\frac{h^{2}}{\beta_n^{-2}(b_{\beta_n}-d_{\beta_n}^{x_{\beta_n}}(0))^2}+\int_{0}^{q_{\beta_n}}\frac{\xi''(s)}{\beta_n^{-2}(b_{\beta_n}-d_{\beta_n}^{x_{\beta_n}}(s))^2}ds\Bigr)=\lim_{n\rightarrow\infty}q_{\beta_n}=1.
 		\end{align*}
 		This completes the proof for $\delta_0>0$.
 		
 		 To check $(L_0,\alpha_0)\in\mathcal{K},$ note that from the proof of Lemma \ref{add:lem2}, $\lim_{n\rightarrow\infty}q_{\beta_n}=1$. Therefore, for any fixed $q\in(0,1),$ 
 		 \begin{align*}
 		 \int_0^1\beta_nx_{\beta_n}(s)ds&=\int_0^q\beta_nx_{\beta_n}(s)ds+\int_q^1\beta_nx_{\beta_n}(s)ds\\
 		 &\geq \int_0^q\beta_nx_{\beta_n}(s)ds+\beta_n(1-q_{\beta_n})
 		 \end{align*}
 		 for sufficiently large $n.$ Using \eqref{eq-1} and the dominated convergence theorem gives
 		 \begin{align*}
 		 L_0\geq \int_0^q\alpha_0(s)ds+\delta_0.
 		 \end{align*}
 		 Since this holds for all $q\in(0,1),$ sending $q\rightarrow 1$ leads to
 		 \begin{align*}
 		 L_0\geq \int_0^1\alpha_0(s)ds+\delta_0
 		 \end{align*}
 		 Thus, $(L_0,\alpha_0)\in\mathcal{K}$ because $\delta_0>0.$
 	\end{proof}

 	Recall the Crisanti-Sommers functional $\mathcal{Q}_\beta$ from \eqref{csf}, 
 	\begin{align*}
 	\mathcal{Q}_\beta(x)&=\frac{1}{2}\Bigl(\int_0^1(\xi_\beta'(q)+h_\beta^2)x (q)dq+\int_0^{\hat{q}}\frac{dq}{\hat{x} (q)}+\log(1-\hat{q})\Bigr)
 	\end{align*}
 	for $x\in\mathcal{M}$, where $\hat{x}(q)=\int_q^1x(s)ds$ and $\hat{q}<1$ satisfies $x(\hat{q})=1.$  Define
 	$$
 	\check{x}(q)=\int_0^qx(s)ds.
 	$$ 
 	Using integration by part for the first integral of $\mathcal{Q}_\beta$ and adapting the notation $\check{x}$, we can express
 	\begin{align}
 	\label{add:eq0}
 	\mathcal{Q}_\beta(x)&=\frac{1}{2}\Bigl((\xi_\beta'(1)+h_\beta^2)\check{x} (1)-\int_0^1\xi_\beta''(s)\check{x} (q)dq+\int_0^{\hat{q}}\frac{dq}{\check{x} (1)-\check{x} (q)}+\log(1-\hat{q})\Bigr).
 	\end{align}

 	\begin{lemma}
 		\label{add:lem1}
 		Let $(y_\beta)\subset {\mathcal{M}}$ such that $(\beta y_\beta)$ converges vaguely to some $\alpha$ on $[0,1)$ and
 		\begin{align}
 		\label{add:lem1:eq1}
 		L:=\lim_{\beta\rightarrow\infty}\int_0^1\beta y_\beta(s)ds<\infty.
 		\end{align}
 		Assume that there exists $v_\beta\in[0,1)$ with $y_\beta(v_\beta)=1$ for all $\beta>0$ such that
 		\begin{align}
 		\label{add:lem1:eq3}
 		\delta:=\lim_{\beta\rightarrow\infty}\beta(1-v_\beta)\in (0,\infty).
 		\end{align} 
 		Then
 		\begin{align*}
 		\lim_{\beta\rightarrow\infty}\frac{\mathcal{Q}_\beta(y_\beta)}{\beta}&=\mathcal{Q}(L,\alpha)=\frac{1}{2}\Bigl((\xi'(1)+h^2)L-\int_0^1\xi''(q)\Bigl(\int_0^q\alpha(s)ds\Bigr)dq+\int_0^1\frac{dq}{L-\int_0^q\alpha(s)ds}\Bigr).
 		\end{align*}
 	\end{lemma}

 	\begin{proof}
 		Since 
 		\begin{align}\label{add:eq1}
 		\beta\check{y}_\beta(q)\leq \int_0^1\beta y_\beta(q)dq,
 		\end{align}
 		the assumption \eqref{add:lem1:eq1} and the bounded convergence theorem imply that
 		\begin{align}
 		\label{add:eq3}
 		\int_0^q \alpha(s)ds\leq L,\,\,\forall q\in[0,1)
 		\end{align}
 		and
 		\begin{align}\label{add:eq4}
 		\lim_{\beta\rightarrow\infty}\frac{1}{\beta}\int_0^1\xi_\beta''(q)\check{y}_\beta (q)dq=\int_0^1\xi''(q)\Bigl(\int_0^q\alpha(s)ds\Bigr)dq.	
 		\end{align}
 		On the other hand, since 
 		\begin{align}\label{add:eq2}
 		\inf_{q\in[0,v_\beta]}\beta(\check{y}_\beta(1)-\check{y}_\beta(q))\geq \beta(\check{y}_\beta(1)-\check{y}_\beta(v_\beta))=\beta(1-v_\beta),
 		\end{align}
 		the condition \eqref{add:lem1:eq3} leads to $L-\int_0^q\alpha(s)ds\geq \delta$ for all $q\in[0,1)$. Using this inequality, \eqref{add:lem1:eq3} and \eqref{add:eq2}, it follows that
 		\begin{align*}
 		&\Bigl|\int_0^{v_\beta}\Bigl(\frac{1}{\beta(\check{y}_\beta (1)-\check{y}_\beta (q))}-\frac{1}{L-\int_0^q\alpha(s)ds}\Bigr)dq\Bigr|\\
 		&\leq \frac{1}{\delta\beta(1-v_\beta)}\int_0^{1}\Bigl(|\beta\check{y}_\beta(1)-L|+1_{[0,v_\beta]}(q)\Bigl|\int_0^q\beta y_\beta(s)ds-\int_0^q\alpha(s)ds\Bigr|\Bigr) dq\rightarrow 0,
 		\end{align*}
 		where the last limit used \eqref{add:eq1}, \eqref{add:eq3} and the bounded convergence theorem. Thus, 
 		\begin{align}\label{add:lem1:proof:eq1}
 		\lim_{\beta\rightarrow\infty}\int_0^{v_\beta} \frac{1}{\check{y}_\beta (1)-\check{y}_\beta (q)}dq=\int_0^1\frac{1}{L-\int_0^q\alpha(s)ds}dq.
 		\end{align}
 		Finally, note that from \eqref{add:lem1:eq3},
 		\begin{align*}
 		\frac{\log(1-v_\beta)}{\beta}=\frac{\log\beta(1-v_\beta)}{\beta}-\frac{\log \beta}{\beta}\rightarrow 0.
 		\end{align*}
 		From \eqref{add:eq0}, this combined with \eqref{add:eq4} and \eqref{add:lem1:proof:eq1} leads to the announced result.
 	\end{proof}

 	\subsection{Proof of Theorems \ref{add:thm1} and \ref{add:thm:char}}\label{subs2.2}
 	
 	We start with a lemma, which states that one can compute the ground state energy through the free energy by sending $\beta$ to infinity.
 	
 		\begin{lemma}\label{lem:discrete}
 			We have that 
 			\begin{align*}
 			GS=\lim_{\beta\rightarrow\infty}\frac{F(\beta)}{\beta}.
 			\end{align*}
 		\end{lemma}
 		
 		It is easy to see that the same result is also valid in the Ising mixed $p$-spin models, see \cite[Section 1.1]{Pan}. In the spherical models, the corresponding proof requires a covering procedure for the sphere via the Dudley entropy integral \cite[Equation $(1.5)$]{Tal14}. As this part of the argument is quite standard (see the proof of Lemma \ref{lem:dis_correlated_field}), we will omit the proof.	
 
 	\begin{proof}[\bf Proof of Theorem \ref{add:thm1}]
 		Consider $(L,\alpha)\in\mathcal{K}.$ Denote $\delta=L-\int_0^1\alpha(s)ds>0$ and take $v_\beta=1-\delta/\beta.$ Consider $y_\beta\in {\mathcal{M}}$ defined by $y_\beta(s)=1$ on $[v_\beta,1]$ and $y_\beta(s)=\min(\alpha(s)/\beta,1)$ on $[0,v_\beta).$ It is easy to see that $\lim_{\beta\rightarrow\infty}\beta(1-v_\beta)=\delta$ and that $(\beta y_\beta)$ converges vaguely to $\alpha$ on $[0,1)$ with
 		\begin{align*}
 		\int_0^1\beta y_\beta(s)ds&=\beta(1-v_\beta)+\int_0^{v_\beta}\min(\alpha(s),\beta)ds\rightarrow \delta+\int_0^1\alpha(s)ds=L.
 		\end{align*}
 		Therefore, from Lemmas \ref{add:lem1}, \ref{lem:discrete} and the Crisanti-Sommers formula \eqref{CS:eq1},
 		\begin{align*}
 		GS\leq \lim_{\beta\rightarrow\infty}\frac{{\mathcal{Q}}_\beta(y_\beta)}{\beta}=\mathcal{Q}(L,\alpha).
 		\end{align*}
 		Since this is true for any $\alpha$ and $L>\int_0^1\alpha(s)ds,$ we get that
 		\begin{align}\label{add:thm1:proof:eq4}
 		GS\leq \inf_{(L,\alpha)\in\mathcal{K}}\mathcal{Q}(L,\alpha).
 		\end{align}
 		
 		Next, we establish the lower inequality for $GS$. Recall $\alpha_0 $ and $L_0$ from \eqref{eq-5}. Note that $(L_0,\alpha_0)\in\mathcal{K}$ by Lemma \ref{add:lem0}. From \eqref{add:lem:eq1}, $$\beta_n\check{x}_{\beta_n}(q)\leq C_\xi',\,\,\forall n\geq 1,q\in[0,1]$$
 		and from \eqref{eq-1}, the weak convergence of $(\beta_nx_{\beta_n})_{n\geq 1}$ and the definition of $L_0$ imply
 		\begin{align*}
 		\lim_{n\rightarrow\infty}\beta_n\check{x}_{\beta_n}(q)&=\int_0^q\alpha_0(s)ds,\,\,q\in[0,1),\\
 		\lim_{n\rightarrow\infty}\beta_n\check{x}_{\beta_n}(1)&=L_0.
 		\end{align*}
 		In addition, from Lemma \ref{add:lem0}, $$
 		\lim_{n\rightarrow\infty}\frac{\log(1-q_{\beta_n})}{\beta_n}=\lim_{n\rightarrow\infty}\Bigl(\frac{\log \bigl(\beta_n(1-q_{\beta_n})\bigr)}{\beta_n}-\frac{\log \beta_n}{\beta_n}\Bigr)=0.
 		$$
 		From these and the expression of $\mathcal{Q}_\beta$ in \eqref{add:eq0}, applying the Fatou lemma and the bounded convergence theorem implies 
 		\begin{align*}
 		GS&=\lim_{n \rightarrow\infty}\frac{{\mathcal{Q}}_\beta(x_{\beta_n})}{\beta_n}\\
 	    &\geq \frac{1}{2}\Bigl((\xi'(1)+h^2)\lim_{n\rightarrow\infty}\beta_n\check{x}_{\beta_n}(1)-\int_0^1\xi''(s)\lim_{n\rightarrow\infty}\beta_n\check{x}_{\beta_n} (q)dq\\
 	    &\quad+\int_0^{1}\lim_{n\rightarrow\infty}\frac{1_{[0,q_\beta]}(q)dq}{\beta_n\bigl(\check{x}_{\beta_n} (1)-\check{x}_{\beta_n} (q)\bigr)}+\lim_{n\rightarrow\infty}\frac{\log(1-{q}_{\beta_n})}{\beta_n}\Bigr)\\
 		&=\mathcal{Q}(L_0,\alpha_0),
 		\end{align*}
 		which implies $$
 		GS\geq \inf_{(L,\alpha)\in\mathcal{K}}\mathcal{Q}(L,\alpha).
 		$$
 		This together with \eqref{add:thm1:proof:eq4} gives \eqref{add:thm1:eq1} and implies that $(L_0,\alpha_0)$ is a minimizer of $\mathcal{Q}$. Noting that $\mathcal{Q}$ is strictly convex on the convex space $\mathcal{K}$, the uniqueness of this minimizer follows. 
 		
 		Finally, we prove \eqref{add:eq6}. To see this, observe that if $(\beta_n')_{n\geq 1}$ is any nonnegative sequence with $\lim_{n\rightarrow\infty}\beta_n'=\infty$ such that the following limits exist,
 		\begin{align}
 		\begin{split}\label{add:eq7}
 		 		\alpha_0'&:=\lim_{n\rightarrow\infty}\beta_n'x_{\beta_n'}\in\mathcal{N}\,\,\mbox{vaguely},\\
 		\delta_0'&:=\lim_{n\rightarrow\infty}\beta_n'(1-q_{\beta_n'}),\\
 		L_0'&:=\lim_{n\rightarrow\infty}\int_0^1\beta_n'x_{\beta_n'}(s)ds,
 		\end{split}
 		\end{align}
 		then the same procedure as the proof for the lower bound of $GS$ leads to $$
 		(L_0',\alpha_0')\in\mathcal{K}\,\,\mbox{and}\,\,GS\geq \mathcal{Q}(L_0',\alpha_0').$$
 		In other words, $(L_0',\alpha_0')$ is a minimizer of $\mathcal{Q}$. Consequently, by the uniqueness of the minimizer, 
 		\begin{align}
 		\label{add:eq8}
 		(L_0',\alpha_0')=(L_0,\alpha_0).
 		\end{align}
 		If now $\int_0^1\beta x_\beta(s)ds$ does not converge to $L_0$ as $\beta\rightarrow\infty,$ then one can find  a sequence $(\beta_n')_{n\geq 1}$ (by using Lemmas \ref{add:lem} and \ref{add:lem2} as well as the argument after Lemma \ref{add:lem2}) such that the three limits \eqref{add:eq7} exist, but $L_0'\neq L_0.$ However, this contradicts \eqref{add:eq8}. Similarly, if $(\beta x_\beta)_{\beta>0}$ does not converge to $\alpha_0$ vaguely on $[0,1)$, then from the equivalence of the weak convergence stated in Remark \ref{rm1}, we can again find a sequence $(\beta_n')_{n\geq 1}$ such that the three limits in \eqref{add:eq7} exist, but 
 		$$
 		\lim_{n\rightarrow\infty}\int_0^1f(s)d\bigl(\beta_n'x_{\beta_n'}\bigr)=\int_0^1f(s)d\alpha_0'\neq \int_0^1f(s)d\alpha_0
 		$$
 		for some continuous function $f$ on $[0,1)$ with compact support. This leads to a contradiction of \eqref{add:eq8} again.
 		As a summary, we show that \eqref{add:eq6} must hold and this ends our proof.
 	\end{proof}
 	
 	\begin{remark}\rm
 		Recall that originally $\alpha_0,\delta_0,L_0$ are defined along a common sequence $(\beta_n)_{n\geq 1}.$ In the above proof, we use the strict convexity of $\mathcal{Q}$ to show that $\alpha_0=\lim_{\beta\rightarrow\infty}\beta x_\beta$ vaguely and $L_0=\lim_{\beta\rightarrow\infty}\int_0^1\beta x_\beta(s)ds$. Nonetheless, these convergences do not guarantee that $\lim_{\beta\rightarrow\infty}\beta(1-q_\beta)$ exists or it is equal to $\delta_0$. 
 	\end{remark}

 	\begin{proof}[\bf Proof of Theorem \ref{add:thm:char}]
 		Let $(L,\alpha)$ and $(L',\alpha')$ be any two pairs in $\mathcal{K}.$ Denote by $\nu$ and $\nu'$ the measures on $[0,1)$ induced by $\alpha$ and $\alpha'$ respectively. Define $L_\theta=(1-\theta)L+\theta L'$ and $\alpha_\theta=(1-\theta)\alpha+\theta\alpha'$ for any $\theta\in[0,1].$ By the minimality of $(L,\alpha)$, we compute the right derivative of $\mathcal{Q}(L_\theta,\alpha_\theta)$ with respect to $\theta$ at $0$,
 		\begin{align*}
 		\partial_\theta\mathcal{Q}(L_\theta,\alpha_\theta)\Big|_{\theta=0}=&(L'-L)\left((\xi'(1)+h^2)-\int_0^1\frac{dq}{\bigl(L-\int_0^q\alpha(r)dr\bigr)^2}\right)\\
 		&+\int_0^1\Bigl[\Bigl(\frac{1}{\bigl(L-\int_0^q\alpha(r)dr\bigr)^2}-\xi''(q)\Bigr)\int_0^q(\alpha'(s)-\alpha(s))ds \Bigr]dq.
 		\end{align*}
 		Suppose now $(L,\alpha)$ is the minimizer. Then
 		\begin{align*}
 		\partial_\theta\mathcal{Q}(L_\theta,\alpha_\theta)\Big|_{\theta=0}\geq 0.
 		\end{align*}
 		Since this true for any $(L',\alpha'),$ the equation \eqref{add:thm:char:eq1} follows
 		and
 		\begin{align*}
 		\int_0^1\Bigl[\Bigl(\frac{1}{\bigl(L-\int_0^q\alpha(r)dr\bigr)^2}-\xi''(q)\Bigr)\int_0^q(\alpha'(s)-\alpha(s))ds \Bigr]dq\geq 0.
 		\end{align*}
 		Interchanging the integrals $ds$ and $dq$, we obtain
 		\begin{align*}
 		\int_0^1\Bigl[\int_s^1\Bigl(\frac{1}{\bigl(L-\int_0^q\alpha(r)dr\bigr)^2}-\xi''(q)\Bigr)dq\Bigr](\alpha'(s)-\alpha(s)) ds\geq 0.
 		\end{align*}
 		Here, using \eqref{add:thm:char:eq1}, 
 		\begin{align*}
 		\int_s^1\Bigl(\frac{1}{\bigl(L-\int_0^q\alpha(r)dr\bigr)^2}-\xi''(q)\Bigr)dq&=\xi'(s)+h^2-\int_0^s\frac{1}{\bigl(L-\int_0^q\alpha(r)dr\bigr)^2}dq=\bar{g}(s)
 		\end{align*} 
 		and thus, recalling $\bar{g}$ from \eqref{add:thm:char:eq2},
 		\begin{align*}
 		\int_0^1\bar{g}(s)(\alpha'(s)-\alpha(s)) ds\geq 0.
 		\end{align*}
 		Next writing $\alpha(s)=\int_0^s\nu(du)$ and $\alpha'(s)=\int_0^s\nu'(du)$ and interchanging the integrals $du$ and $ds$, the last inequality becomes
 		\begin{align}
 		\begin{split}
 		0&\leq \int_0^1\bar{g}(s)(\alpha'(s)-\alpha(s)) ds\\\notag
 		&=\int_0^1\bar{g}(s)\Bigl(\int_0^s(\nu'(du)-\nu(du))\Bigr)ds\\\notag
 		&=\int_0^1\Bigl(\int_u^1\bar{g}(s)ds\Bigr)(\nu'(du)-\nu(du))\notag
 		\end{split}\\
 		\begin{split}\label{add:thm:char:proof:eq1}
 		&=\int_0^1g(u)(\nu'(du)-\nu(du)),
 		\end{split}
 		\end{align}
 		where $g$ is defined through \eqref{add:thm:char:eq3}. If we take $\nu'=\nu+\nu''$ for any arbitrary finite measure $\nu''$ on $[0,1),$ then we obtain that
 		$0\leq \int_0^1g(u)\nu''(du),$
 		which implies 
 		\begin{align}
 		\label{add:thm:char:proof:eq2}
 		\min_{u\in[0,1]}g(u)\geq 0.
 		\end{align}
 		On the other hand, if we let $\nu'$ be the measure defined by $\nu'(A)=\nu(A\cap \{u\in[0,1):g(u)\leq 0\})$, then \eqref{add:thm:char:proof:eq1} gives
 		\begin{align*}
 		0&\geq \int_{\{u\in[0,1):g(u)>0\}} g(u)\nu(du),
 		\end{align*}
 		which together with \eqref{add:thm:char:proof:eq2} means that $g(u)=0$ for all $u$ in the support of $\nu$ and thus $\nu(S)=\nu([0,1)).$ Conversely, it is clear that if \eqref{add:thm:char:eq1}, \eqref{add:thm:char:proof:eq2} and $\nu(S)=\nu([0,1))$ hold, then \eqref{add:thm:char:proof:eq1} is valid and thus, $(L,\alpha)$ is the minimizer by uniqueness.
 	\end{proof}
 	
 	\subsection{Proof of Propositions \ref{add:prop1}, \ref{add:prop1.5}, and \ref{add:prop2}}\label{subs2.3}
 	We now use Theorems \ref{add:thm1} and \ref{add:thm:char} to compute the ground state energy of the three cases of the mixed $p$-spin models mentioned in Propositions \ref{add:prop1}, \ref{add:prop1.5}, and \ref{add:prop2}.

 	\begin{proof}[\bf Proof of Proposition \ref{add:prop1}]
 		Assume $\xi'(1)+h^2\geq \xi''(1).$ Let $L=(\xi'(1)+h^2)^{-1/2}$ and $\alpha=0$. It can be immediately seen that \eqref{add:thm:char:eq1} is fulfilled by this pair $(L,\alpha)$. In addition,
 		\begin{align*}
 		\bar{g}(s)&=\xi'(s)+h^2-s(\xi'(1)+h^2),\,\,\bar{g}(1)=0,\\
 		\bar{g}'(s)&=\xi''(s)-(\xi'(1)+h^2)\leq 0.
 		\end{align*}
 		This implies that $g(u)> 0$ for $u\in(0,1)$. Let $\nu$ be the measure induced by $\alpha.$ Since $S=\emptyset$ and $\nu(S)=0=\nu([0,1))$, we conclude that $(L,\alpha)$ is the minimizer of \eqref{add:thm1:eq1}, which means that the model is replica symmetric at zero temperature and \eqref{GS:eq1} holds. 
 		
 		Conversely, let  $(L_0,\alpha_0)$ be the minimizer of \eqref{add:thm1:eq1}. Suppose that the model is replica symmetric at zero temperature, i.e., $\alpha_0=0.$ From \eqref{add:thm:char:eq1}, $L_0=(\xi'(1)+h^2)^{-1/2}$ and $\bar{g}(1)=0$. If $\xi'(1)+h^2<\xi''(1)$, then there exists some $s_0\in(0,1)$ such that
 		$\bar{g}'(s)>0$ for all $s\in[s_0,1]$, from which $\bar{g}(s)<0$ for $s\in[s_0,1)$ and thus, $g(u)<0$ for $u\in[s_0,1).$  However, this contradicts the assumption of $(L_0,\alpha_0)$ being the minimizer of \eqref{add:thm1:eq1} since $\min_{u\in[0,1]}g(u)\geq 0$ by Theorem \ref{add:thm:char}. 
 	\end{proof}
 		
 	\begin{proof}[\bf Proof of Proposition \ref{add:prop1.5}]
        Note the condition $\xi'(1)+h^2<\xi''(1)$ implies that $\gamma_p>0$ for at least one $p\geq 3,$ which implies that $q\xi''(q)-\xi'(q)$ is strictly increasing on $[0,1]$ and thus, gives the existence and uniqueness of $q_0$ in \eqref{GS:eq3}. Choose $L=\xi''(q_0)^{-1/2}$ and define $\alpha$ by 
 		\begin{align*}
 		\alpha(s)&=
 		\left\{
 		\begin{array}{ll}
 		0,&\mbox{if $q\in[0,q_0)$},\\
 		\frac{\xi'''(s)}{2\xi''(s)^{3/2}},&\mbox{if $q\in[q_0,1)$}.
 		\end{array}
 		\right.
 		\end{align*}
 		Here, $\alpha$ is nondecreasing due to the assumption of $\xi''(s)^{-1/2}$ being concave on $(0,1].$ With this choice of $(L,\alpha)$, a direct computation gives \eqref{add:thm:char:eq1} by applying \eqref{GS:eq3}. On the other hand, since 
 		$$
 		\bar{g}(q)=\left\{
 		\begin{array}{ll}
 		\xi'(s)+h^2-s\xi''(s),&\mbox{if $q\in[0,q_0)$},\\
 		0,&\mbox{if $q\in[q_0,1]$}.
 		\end{array}
 		\right.
 		$$
 		one sees that $g(u)>0$ if $u\in[0,q_0)$ and $g(u)=0$ if $u\in[q_0,1]$, from which $S=[q_0,1)$ and clearly $\nu(S)=\nu([0,1))$ for $\nu$ be the measure induced by $\alpha.$ Hence, $(L,\alpha)$ is the minimizer of  \eqref{add:thm1:eq1} and \eqref{GS:eq2} is obtained by applying \eqref{GS:eq3} to the following computation,
 		\begin{align*}
 		GS&=\frac{1}{2}\Bigl((\xi'(1)+h^2)\xi''(q_0)^{-1/2}-\int_{q_0}^1\xi''(q)\bigl(\xi''(q_0)^{1/2}-\xi''(q)^{1/2}\bigr)dq+q_0\xi''(q_0)^{1/2}+\int_{q_0}^{1}\xi''(q)^{1/2}\Bigr)\\
 		&=\frac{1}{2}\Bigl((\xi'(q_0)+h^2)\xi''(q_0)^{-1/2}+q_0\xi''(q_0)^{1/2}+2\int_{q_0}^{1}\xi''(q)^{1/2}\Bigr)\\
 		&=q_0\xi''(q_0)^{1/2}+\int_{q_0}^1 \xi''(q)^{1/2}dq.
 		\end{align*} 	
 	\end{proof}

 	\begin{proof}[\bf Proof of Proposition \ref{add:prop2}]
 		Recall the constant $z$ from \eqref{add:prop2:eq2}. Set $L=(z\delta)^{1/2}+(z^{-1}\delta )^{1/2}$ and $\alpha(s)=(z\delta)^{1/2}$ for all $s\in[0,1)$
 		for $\delta:=z(1+z)^{-1}.$ Equation \eqref{add:thm:char:eq1} follows by a straightforward computation,
 		\begin{align*}
 		\int_0^1\frac{dq}{\bigl(L-\int_0^q\alpha(s)ds\bigr)^2}&=\frac{1}{z\delta}\int_0^1\frac{dq}{\bigl(z^{-1}+1-q\bigr)^2}=\frac{z}{\delta(1+z)}=1=\xi'(1).
 		\end{align*}
 		Since
 		\begin{align*}
 		\bar{g}(s)&=s^{p-1}-\frac{1}{z\delta}\int_0^s\frac{dq}{(z^{-1}+1-q)^2}\\
 		&=s^{p-1}+\frac{1}{z}-\frac{1}{\delta\bigl(z(1-s)+1\bigr)},\,\,s\in[0,1],
 		\end{align*}
 		we have that
 		\begin{align*}
 		g(u)&=\int_u^1\bar{g}(s)ds=\frac{1}{p}(1-u^{p})+\frac{1-u}{z}-\frac{1+z}{z^2}\log\bigl(z(1-u)+1\bigr).
 		\end{align*}
 		Clearly $g(1)=0$ and by \eqref{add:prop2:eq2}, $g(0)=0$. 
 		
 		We now proceed to check that $g(u)> 0$ for all $u\in(0,1).$ To see this, note that 
 		\begin{align*}
 		g'(u)&=\frac{u}{1+z(1-u)}-u^{p-1}=\frac{u}{1+z(1-u)}a(u),
 		\end{align*}
 		where 
 		$$
 		a(u)=1-u^{p-2}(1+z)+zu^{p-1}.
 		$$
 		It is clear that $g'(u)>0$ for $u$ being sufficiently close to $0.$ Let $u_0\in(0,1)$ be a local maximum of $g$. If $g(u)<0$ for some $u\in(0,1),$ then $g$ has a local minimum at some $u_1\in(0,1).$ Consequently, $a(u_0)=a(u_1)=a(1)=0$ and by Rolle's theorem, there exist distinct $v_0,v_1\in(0,1)$ such that 
 		\begin{align*}
 		a'(v_0)&=v_0^{p-3}\bigl(z(p-1)v_0-(p-2)(1+z)\bigr)=0,\\
 		a'(v_1)&=v_1^{p-3}\bigl(z(p-1)v_1-(p-2)(1+z)\bigr)=0,
 		\end{align*}
 		but these equations also imply that $v_0=v_1$, a contradiction. Therefore, $g(u)>0$ for all $u\in(0,1).$ Consequently, $S=\{0\}$ and $\nu(S)=\nu([0,1))$, from which $(L,\alpha)$ is the minimizer to \eqref{add:thm1:eq1} by Theorem~\ref{add:thm1} and thus, $GS$ equals
 		\begin{align*}
 		&\frac{1}{2}\Bigl((z\delta)^{1/2}+(z^{-1}\delta)^{1/2}-(p-1)(z\delta)^{1/2}\int_0^1 q^{p-1}dq+\frac{1}{(z\delta)^{1/2}}\int_0^1\frac{dq}{z^{-1}+1-q}\Bigr)\\
 		&=\frac{1}{2}\Bigl((z^{-1}\delta)^{1/2}+\frac{(z\delta)^{1/2}}{p}+\frac{\log(z+1)}{(z\delta)^{1/2}}\Bigr),
 		\end{align*}
 		which gives \eqref{add:prop2:eq1} by substituting $\delta=z(1+z)^{-1}$ and using the equation \eqref{add:prop2:eq2}.
 	\end{proof}

	\section{Proof of chaos in disorder}\label{sec:dis}
	
    Throughout this section, we assume that the mixed $p$-spin model is even, i.e., $\gamma_p=0$ for all odd $p\geq 3.$ We will establish disorder chaos for the ground state. Recall the Hamiltonians $H_{N,t}^1$ and $H_{N,t}^2$ from \eqref{hamilton} and their maximizers $\sigma_t^*$ and $\tau_t^*$ from \eqref{more:eq4}. 	Define the coupled Hamiltonian and the product measure on $S_N\times S_N$ by
    \begin{align}
    \begin{split}\label{hamilton2}
    H_{N,t}(\sigma,\tau)=H_{N,t}^1(\sigma)+H_{N,t}^2(\tau),\\
    d m_N(\sigma,\tau)=dm_N(\sigma)\times dm_N(\tau).
    \end{split}
    \end{align} 
    
    To facilitate our argument, we give an outline of the approach.  While the Crisanti-Sommers functional $\mathcal{Q}_\beta$ was heavily used in Section $2$, we shall adapt the Parisi function $\mathcal{P}_\beta$ in \eqref{prop3:proof:eq0} throughout this section. Our main tool is based on Guerra's replica symmetry breaking (RSB) bound for the free energy of $H_{N,t}(\sigma,\tau)$ with overlap $R(\sigma,\rho)$ staying around a given level $u\in[-1,1].$ This bound should be understood as a two-dimensional generalization of the Parisi functional $\mathcal{P}_\beta$. It was originally derived in \cite{Chen151} to investigate disorder chaos for the spherical mixed even $p$-spin model at positive temperature. In the inverse temperature limit, we show that for any $u\in(-1,1)$, this bound leads to a RSB bound (see Theorem \ref{lem5}) for 
    \begin{align}\label{more:eq8}
   \limsup_{N\rightarrow\infty}\e \max_{|R(\sigma,\tau)-u|\geq \varepsilon}\frac{H_{N,t}(\sigma,\tau)}{N}
    \end{align}
    for $\varepsilon>0$ sufficiently small. For any $t\in(0,1)$, by a careful control of the RSB bound, we prove that there exists a unique constant $u_t$ such that \eqref{more:eq8} is strictly less than $2GS$ when $u=u_t.$ Consequently, this combined with an application of the Gaussian concentration inequality implies that for large enough $N$ the following inequality holds with overwhelming probability,
    $$
    \max_{|R(\sigma,\tau)-u_t|\geq \varepsilon}\frac{H_{N,t}(\sigma,\tau)}{N}<\max_\sigma\frac{H_{N,t}^1(\sigma)}{N}+\max_{\tau}\frac{H_{N,t}^2(\tau)}{N}.
    $$
    As a result, from the optimality of $\sigma_t^*,\tau_t^*$, the event $|R(\sigma_t^*,\tau_t^*)- u_t|\leq \varepsilon$ holds with overwhelming probability, which gives the assertion of Theorem~\ref{thm4}.

    In Subsection $3.1$, we state the RSB bound for the coupled free energy from \cite{Chen151}. Subsection $3.2$ will establish the RSB bound for the maximal coupled Hamiltonian \eqref{more:eq8}.   As one will see, similar technicalities appeared in Section $2$ will occur here when we handle the inverse temperature limit of the RSB bound for the coupled free energy. We shall perform the same trick in Section $2$ to bypass these difficulties. In Subsection $3.3$, we will analyze the RSB bound obtained in Subsection~$3.2$ and use it to show that there exists a constant $u_t$ that fulfills the claim properties of Theorem \ref{thm4}. Finally, with the help of Subsection~$3.3$, we deduce that the quantity \eqref{more:eq8} is strictly less than $2GS$ when $u=u_t$ and provide the proof of Theorem \ref{thm4} in Subsection $3.4$.

    	\subsection{RSB bound for the coupled free energy}

    	For any $u\in\mathbb{R}$, set $\mbox{sign}(u)=1$ if $u\geq 0$ and $\mbox{sign}(u)=-1$ if $u<0.$ The Guerra replica symmetry breaking bound for the coupled free energy, associated to the coupled Hamiltonian $H_{N,t}$, is stated as follows.

    	\begin{proposition}\label{prop4}
    		For  $x\in\mathcal{M}$, $\lambda \in \mathbb{R}$ and $b>\int_0^1\xi_{\beta}''(s)x(s)ds+|\lambda |$, we have that for any $u\in [-1,1],$
    		\begin{align*}
    		F_{\beta}(t,u)&:=\lim_{\eta\rightarrow 0} \limsup_{N \rightarrow \infty} \frac{1}{N}\e \log\int_{|R(\sigma,\tau)-u|<\eta}\exp\bigl(\beta H_{N,t}(\sigma,\tau)\bigr)dm_N(\sigma,\tau) \\
    		&\leq  \mathcal{P}_{\beta}(t,u,x,b,\lambda),
    		\end{align*}
    		where the functional $\mathcal{P}_\beta(t,u,x,b,\lambda)$ is defined as follows. Set 
    		\begin{align*}
    		d_{\beta,u}^x(q)=d_\beta^x(|u|)+\frac{1-t}{1+t}(d_\beta^x(q)-d_\beta^x(|u|)),\,\,\forall q\in[0,|u|].
    		    		\end{align*}
    		Define
    		\begin{align}\label{prop4:eq1}
    		\mathcal{P}_\beta(t,u,x,b,\lambda) &=T_\beta(t,u,x,b,\lambda)+\left\{
    		\begin{array}{ll}
    		\frac{h_\beta^2}{b-\lambda- d_\beta^x(0)},&\mbox{if $u\in[0,1]$},\\
    		\frac{h_\beta^2}{b-\lambda-d_{\beta,u}^x(0)},&\mbox{if $u\in[-1,0]$},
    		\end{array}\right.
    		\end{align}
    		where for $\iota:=\mbox{sign}(u)$,
    		\begin{align}
    		\begin{split}\label{prop4:eq2}
    		T_\beta(t,u,x,b,\lambda)&:=\log \sqrt{\frac{b^2}{b^2-\lambda^2}}+\frac{1+t}{2}  \int_0^{|u|} \frac{\xi_{\beta}''(q)}{b-\iota \lambda-d_\beta^x(q)} dq+ \frac{1-t}{2}  \int_0^{|u|} \frac{\xi_{\beta}''(q)}{b+\iota \lambda-d_{\beta,u}^x(q)} dq\\
    		&\quad +\frac{1}{2} \int_{|u|}^1 \frac{\xi_{\beta}''(q)}{b-\lambda-d_\beta^x(q)} dq + \frac{1}{2} \int_{|u|}^1 \frac{\xi_{\beta}''(q)}{b+\lambda-d_\beta^x(q)} dq \\
    		&\quad-\lambda u +b-1- \log b - \int_0^1 q\xi_{\beta}''(q) x(q)dq.
    		\end{split}
    		\end{align}
    	\end{proposition}
    	
    		This proposition is taken from Proposition 2.1 in \cite{Chen151}. To see its derivation and how it was used to prove disorder chaos at positive temperature, we invite the readers to check \cite{Chen151}.
     

    \subsection{RSB bound for the maximal coupled Hamiltonian}
    
    In this subsection, we derive a RSB bound for the maximal coupled Hamiltonian based on  Proposition \ref{prop4}.
    Recall that $(x_\beta,b_\beta)$ stands for the minimizer to the Parisi formula \eqref{prop3:proof:eq0}. Also, recall the sequence $(\beta_n)_{n\geq 1}$ and the quantities $\alpha_0,\delta_0,L_0$ from \eqref{eq-5}. From the second and third inequality of \eqref{add:lem:eq1}, we may assume without loss of generality that the following two sequences converge along the same sequence $(\beta_n)_{n\geq 1},$ \begin{align}
    \begin{split}\label{more:eq9}
    V_0&:=\lim_{n\rightarrow\infty}\int_0^1\xi''(s)\beta_n x_{\beta_n}(s)ds,\\
    V_1&:=\lim_{n\rightarrow\infty}\int_0^1s\xi''(s)\beta_n x_{\beta_n}(s)ds.
    \end{split}
    \end{align}
    Note that $V_0,V_1\leq C_\xi'$ by Lemma \ref{add:lem2}.
    Set
    \begin{align}\label{more:eq1}
    B&=\xi''(1)\delta_0+\delta_0^{-1}
    \end{align}
    and define
    \begin{align}
    \begin{split}\label{more:eq2}
    D(q)&=\left\{
    \begin{array}{ll}
    V_0-\int_0^q\xi''(s)\alpha_0(s)ds,&\mbox{if $q\in[0,1)$},\\
    0,&\mbox{if $q=1$}.
    \end{array}\right.
    \end{split}
    \end{align}
    A crucial fact that will be shown in Lemma \ref{prop3} below is the inequality,
    \begin{align*}
    0\leq D(q)\leq D(0)<B,\,\,\forall q\in [0,1].
    \end{align*}
      For any $r,r'\in(0,1)$ and $u\in[-1,1]$, consider two sets
      \begin{align}
      \begin{split}\label{thm1:eq2}
      A_r(u)&:=\{v\in[-1,1]:|u-v|\geq r\},\\
      \mathcal{A}_{r}(u)&:=\bigl\{(\sigma,\tau)\in S_N\times S_N:R(\sigma,\tau)\in A_r(u)\bigr\}.
      \end{split}
      \end{align}
      Our RSB bound is stated as follows.

      \begin{theorem}\label{lem5}
      	Let $\Lambda$ be a measurable function on $[-1,1]$ with $\|\Lambda\|_\infty<(B-D(0))/2.$ For any $t\in(0,1)$, $u\in(-1,1)$ and $0<\varepsilon<\min(1+u,1-u)$, we have
      	\begin{align}\label{lem5:eq1}
      	\limsup_{N\rightarrow\infty}\e\max_{\mathcal{A}_{\varepsilon}(u)}\frac{H_{N,t}(\sigma,\tau)}{N}&\leq \sup_{v\in A_{\varepsilon/2}(u)}E(t,v,\Lambda(v)),
      	\end{align}
      	where the function $E(t,v,\lambda)$ is defined in the following proposition.
      \end{theorem}

     \begin{proposition}\label{lem2}
     	For any $u\in[-1,1]$, $t\in[0,1]$, and $|\lambda|<B-D(0),$ 
     	\begin{align}\label{lem2:eq1}
     	E(t,u,\lambda):=\lim_{n\rightarrow\infty}\frac{\mathcal{P}_{\beta_n}(t,u,x_{\beta_n},b_{\beta_n},\lambda_{\beta_n})}{\beta_n}
     	\end{align}
     	exists, where $\lambda_{\beta_n}:=\beta_n \lambda.$ The function $E(t,u,\lambda)$ can be computed as		
     	\begin{align}\label{lem2:eq2}
     	E(t,u,\lambda) &=T(t,u,\lambda)+\left\{
     	\begin{array}{ll}
     	\frac{h^2}{B-\lambda-D (0)},&\mbox{if $u\in[0,1]$},\\
     	\frac{h^2}{B-\lambda-D_u (0)},&\mbox{if $u\in[-1,0]$},
     	\end{array}\right.
     	\end{align}
     	where 
     		\begin{align*}
     		T(t,u,\lambda)&:=\frac{1+t}{2}  \int_0^{|u|} \frac{\xi''(q)}{B-\iota\lambda-D (q)} dq+ \frac{1-t}{2}  \int_0^{|u|} \frac{\xi''(q)}{B+\iota\lambda-D_u (q)} dq\\
     		&\quad +\frac{1}{2} \int_{|u|}^1 \frac{\xi''(q)}{B-\lambda-D (q)} dq + \frac{1}{2} \int_{|u|}^1 \frac{\xi''(q)}{B+\lambda-D (q)} dq -\lambda u +\xi''(1)\delta_0+\delta_0^{-1}-V_1
     		\end{align*}
     		for $\iota:=\mbox{sign}(u)$. Here,
     	\begin{align}
     	\label{more:eq3}
     	D_u (q):=\left\{
     	\begin{array}{ll}
     	V_0-\int_0^{|u|}\xi''(s)\alpha_0(s)ds+\frac{1-t}{1+t}\int_q^{|u|}\xi''(s)\alpha_0(s)ds,&\mbox{if $0\leq q\leq |u|<1$},\\
     	\frac{1-t}{1+t}\Bigl(V_0-\int_0^{q}\xi''(s)\alpha_0(s)ds\Bigr),&\mbox{if $0\leq q<|u|=1$},\\
     	0,&\mbox{if $q=|u|=1$},
     	\end{array}
     	\right.
     	\end{align}
     	satisfies $D_u(q)\leq D(0)$ for all $q\in[0,|u|]$.
     \end{proposition}
     
     The RSB bound \eqref{lem5:eq1} will play an essential role in controlling the maximal coupled Hamiltonian. 
     This inequality is obtained by the scaled inverse temperature limit of the functional $\mathcal{P}_{\beta}(t,u,x,b,\lambda)$ along the sequence $(\beta_n).$ With an extra effort, it seems possible that one can derive a two-dimensional generalization of the Crisanti-Sommers functional $\mathcal{Q}_\beta$ for the coupled free energy $F_\beta(t,u)$ and use it to derive an analogous  bound for the maximal coupled Hamiltonian.
         
     Throughout the remainder of this subsection, we establish the proof of Theorem \ref{lem5} and Proposition~\ref{lem2}. We begin with the following lemma.
     
        \begin{lemma}\label{prop3}
        	Recall $B$ and $D(q)$ respectively from \eqref{more:eq1} and \eqref{more:eq2}. They satisfy
        		\begin{align}\label{prop3:eq1}
        		0\leq D(q)\leq D(0)<B,\,\,\forall q\in [0,1].
        		\end{align}
        \end{lemma}
        
        \begin{proof}	
        	Recall the function $d_\beta^x(q)$ from \eqref{parisi}. Applying \eqref{add:thm1:proof:eq0} yields
        	\begin{align*}
        	\lim_{n\rightarrow\infty}\frac{b_{\beta_n}}{{\beta_n}}&=\lim_{n\rightarrow\infty}\Bigl({\beta_n}(\xi'(1)-\xi'(q_{\beta_n}))+\frac{1}{{\beta_n}(1-q_{\beta_n})}\Bigr)=\xi''(1)\delta_0+\delta_0^{-1}=B.
        	\end{align*}
        	For $q\in[0,1),$ by \eqref{eq-1} and the bounded convergence theorem,
        	\begin{align*}
        	\lim_{n\rightarrow\infty}\frac{d_{\beta_n}^{x_{\beta_n}}(q)}{{\beta_n}}&=	\lim_{n\rightarrow\infty}\Bigl(\int_0^1\xi''(s){\beta_n} x_{\beta_n}(s)ds-\int_0^q\xi''(s){\beta_n} x_{\beta_n}(s)ds\Bigr)\\&=V_0-\int_0^q\xi''(s)\alpha(s)ds\\
        	&=D(q)
        	\end{align*}
        	and for $q=1,$
        	\begin{align*}
        	\lim_{n\rightarrow\infty}\frac{d_{\beta_n}^{x_{\beta_n}}(1)}{{\beta_n}}&=0=D(1).
        	\end{align*}
        	To see the strict inequality between $D(0)$ and $B$, we apply \eqref{add:thm1:proof:eq-1} to get
        	\begin{align*}
        	b_{\beta_n}-d_{\beta_n}^{x_{\beta_n}}(q')&=\frac{1}{\int_{q'}^1x_{\beta_n}(s)ds}
        	\end{align*}
        	for $q'$ the minimum of the support of $\mu_{\beta_n}.$ Since $x_{\beta_n}=0$ on $[0,q')$, this equation can be rewritten as
        	\begin{align*}
        	b_{\beta_n}-d_{\beta_n}^{x_{\beta_n}}(0)&=\frac{1}{\int_0^1x_{\beta_n}(s)ds}.
        	\end{align*}
        	Thus, by the definition of $L_0$ and the fact that $L_0>0,$
        	\begin{align}\label{prop3:proof:eq2}
        	B-D(0)&=\lim_{n\rightarrow\infty}\frac{1}{{\beta_n}}\bigl(b_{\beta_n}-d_{\beta_n}^{x_{\beta_n}}(0)\bigr)=\frac{1}{L_0}>0,
        	\end{align}
        	which gives \eqref{prop3:eq1}.
        \end{proof}

        \begin{proof}[\rm \bf Proof of Proposition \ref{lem2}]
        	Recall ${\beta_n}^{-1}\mathcal{P}_{\beta_n}$ and ${\beta_n}^{-1}T_{\beta_n}$ from Proposition \ref{prop4}. Take $$(x,b,\lambda)=(x_{\beta_n},b_{\beta_n},{\beta_n}\lambda).$$ Observe that from \eqref{prop3:proof:eq2} and the monotonicity of $d_{{\beta_n}}^{x_{\beta_n}}$ in $q$, we have that
        	\begin{align}
        	\label{eq.0}
        	b_{\beta_n}-d_{\beta_n}^{x_{\beta_n}}(q)\geq \frac{1}{2L_0}
        	\end{align}
        	for all $q\in[0,1]$ and sufficiently large $n.$ From the dominated convergence theorem, this combined with the limits we obtained in the proof of Lemma \ref{prop3} suggests that we only need to handle the limit of ${\beta_n}^{-1}d_{{\beta_n},u}^{x_{\beta_n}}$ in the equations \eqref{prop4:eq1} and \eqref{prop4:eq2} and ensure that $D_u(q)\leq D(0)$ for all $q\in[0,|u|].$ These can be justified as follows. If $|u|<1,$ we write
        	\begin{align}\label{eq-12}
        	\begin{split}
        	\frac{d_{{\beta_n},u}^{x_{\beta_n}}(q)}{{\beta_n}}&=\int_0^1\xi''(s){\beta_n} x_{\beta_n}(s)ds-\int_0^{|u|}\xi''(s){\beta_n} x_{\beta_n}(s)ds+\frac{1-t}{1+t}\int_{q}^{|u|}\xi''(s){\beta_n} x_{\beta_n}(s)ds\\
        	&\leq \int_0^1\xi''(s){\beta_n} x_{\beta_n}(s)ds
        	\end{split}
        	\end{align}
        	and pass to limit to get
        	\begin{align*}
        	D_u(q)&=V_0-\int_0^{|u|}\xi''(s)\alpha_0(s)ds+\frac{1-t}{1+t}\int_{q}^{|u|}\xi''(s)\alpha_0(s)ds\leq D(0).
        	\end{align*}
        	For $|u|=1,$ we write 
        	\begin{align*}
        	\frac{d_{{\beta_n},u}^{x_{\beta_n}}(q)}{{\beta_n}}&=\frac{1-t}{1+t}\Bigl(\int_0^1\xi''(s){\beta_n} x_{\beta_n}(s)ds-\int_{0}^{q}\xi''(s){\beta_n} x_{\beta_n}(s)ds\Bigr)\leq \int_0^1\xi''(s){\beta_n} x_{\beta_n}(s)ds.
        	\end{align*}
        	From this, passing to limit implies that for $q\in[0,1)$,
        	\begin{align*}
        	D_u(q)&=\frac{1-t}{1+t}\Bigl(V_0-\int_{0}^{q}\xi''(s)\alpha_0(s)ds\Bigr)\leq D(0)
        	\end{align*}
        	and for $q=1,$  
        	$$
        	D_u(1)=0\leq D(0).
        	$$
        	This finishes our proof.
        \end{proof}
        
        Next we turn to the proof of Theorem \ref{lem5}. 
        To derive \eqref{lem5:eq1}, one will control the RSB bound $\mathcal{P}_{\beta_n}(t,v,x,b,\lambda)$ uniformly over all $v\in \mathcal{A}_\varepsilon(u)=[-1,u-\varepsilon]\cup[u+\varepsilon,1]$ as $n\rightarrow\infty.$ Since $(\beta_nx_{\beta_n})_{n\geq 1}$ is only vaguely convergent on $[0,1)$, a more delicate treatment is needed, especially when we deal with the scaling limit of the term $d_{\beta,v}^{x}$ for $|v|$ being very close to $1$  (see the discussion in Remark \ref{rm3} below). 
        We shall control the value of $|v|$ to stay away from $1$ by considering the following two sets,
        \begin{align}
        \begin{split}\label{thm1:eq3}
        A_{r}^{r'}(u)&:=\{v\in[-1+r',1-r']:|u-v|\geq r\},\\
        \mathcal{A}_{r}^{r'}(u)&:=\bigl\{(\sigma,\tau)\in S_N\times S_N:R(\sigma,\tau)\in A_r^{r'}(u)\bigr\}
        \end{split}
        \end{align}
        for $r,r'\in (0,1)$ and $u\in[-1,1].$ We device two technical lemmas.

        \begin{lemma}\label{lem:dis_correlated_field}  
        	For any $\varepsilon,\kappa\in(0,1)$, $t\in[0,1]$ and $u\in[-1,1]$, we have
        	\begin{align}
        	\begin{split}\label{lem:dis_correlated_field:eq1}
        	&\limsup_{N\rightarrow\infty}\e\frac{\max_{ \mathcal{A}_{\varepsilon}^\kappa(u)}H_{N,t}(\sigma,\tau)}{N}\\
        	&\leq \limsup_{\beta\rightarrow\infty}\limsup_{N\rightarrow\infty}
        	\frac{1}{N{\beta}}\e\log\int_{ \mathcal{A}_{\varepsilon/2}^{\kappa/2}(u)}\exp\bigl({\beta} H_{N,t}(\sigma,\tau)\bigr)dm_N(\sigma,\tau).
        	\end{split}
        	\end{align} 	
        \end{lemma}
        
        \begin{lemma}\label{extralem}  
        	Let $u\in(-1,1)$ and $t\in[0,1].$ Assume that 
        	\begin{align}
        	\begin{split}\label{cond}
        	&0<\varepsilon<\min(1+u,1-u),\\
        	&0<\kappa<\min(1+u-\varepsilon,1-u-\varepsilon,1/2).
        	\end{split}
        	\end{align}
        	There exists some constant $C>0$ independent of $N$ such that
        	\begin{align}\label{lem:dis_correlated_field:eq2}
        	\e\frac{\max_{ \mathcal{A}_{\varepsilon}(u)}H_{N,t}(\sigma,\tau)}{N}\leq\e\frac{\max_{ \mathcal{A}_{\varepsilon}^\kappa(u)}H_{N,t}(\sigma,\tau)}{N}+C\kappa^{1/2}.
        	\end{align}
        \end{lemma}
        
        The above two lemmas allow us to control the ground state energy of $H_{N,t}$ restricted to $\mathcal{A}_\varepsilon(u)$ via the coupled free energy. As their proofs are not directly related to our main arguments, we shall defer the details to the appendix.

        \begin{proof}[\rm \bf Proof of Theorem \ref{lem5}] Let $(x_{\beta_n},b_{\beta_n})$ be the minimizer to the Parisi formula \eqref{prop3:proof:eq0}. Recall $E(t,u,\lambda)$ from Proposition~\ref{lem2}. Denote $\lambda_{\beta_n}(v)={\beta_n}\Lambda(v).$  For $r\in(0,1)$ and $v\in[-1,1],$ denote by $\mathcal{A}_{r}^c(v)$ the complement of $\mathcal{A}_{r}(v)$. This proof has three major steps:
        	
        	{\bf Step I:} From Proposition \ref{prop4}, 
        	\begin{align*}
        	&\lim_{\eta\rightarrow 0}\limsup_{N\rightarrow\infty}\frac{1}{N}\e\log \int_{\mathcal{A}_{\eta}^c(v)}\exp\bigl({\beta_n} H_{N,t}(\sigma,\tau)\bigr)dm_N(\sigma,\tau)\leq \mathcal{P}_{\beta_n}(t,v,x_{\beta_n},b_{\beta_n},\lambda_{\beta_n}(v))
        	\end{align*}
        	for all $v\in[-1,1]$. From this, for any given $\delta>0$ and $v\in A_{\varepsilon/2}(u)$, there exist $\eta(v)>0$ and $N(v)\in\mathbb{N}$ such that
        	\begin{align}\label{lem5:proof}
        	\frac{1}{N}\e\log \int_{\mathcal{A}_{\eta}^c(v)}\exp\bigl({\beta_n} H_{N,t}(\sigma,\tau)\bigr)dm_N(\sigma,\tau)\leq \mathcal{P}_{\beta_n}(t,v,x_{\beta_n},b_{\beta_n},\lambda_{\beta_n}(v))+\delta
        	\end{align}
        	for all $0<\eta\leq \eta(v)$ and $N\geq N(v).$ Consider
        	$$
        	0<\kappa<\min(1+u-\varepsilon,1-u+\varepsilon).
        	$$
        	Note that the quantities $u,\varepsilon,\kappa$ satisfy the condition \eqref{cond}.
        	Since $A_{\varepsilon/2}^{\kappa/2}(u)$ is compact, we can find $v_1,\ldots,v_k\in A_{\varepsilon/2}^{\kappa/2}(u)$ such that $I_{j}:=(v_j-\eta(v_j),v_j+\eta(v_j))$ for $j=1,\ldots,k$ form an open covering of $A_{\varepsilon/2}^{\kappa/2}(u)$. Note that 
        	$$
        	\{(\sigma,\tau)\in S_N\times S_N:R(\sigma,\tau)\in I_j\}=\mathcal{A}_{\eta(v_j)}^c(v_j).
        	$$
        	Using the open covering $I_1,\ldots,I_k$, we have that
        	\begin{align*}
        	&\frac{1}{N}\e\log \int_{\mathcal{A}_{\varepsilon/2}^{\kappa/2}(u)}\exp \bigl({\beta_n} H_{N,t}(\sigma,\tau)\bigr)dm_N(\sigma,\tau)\\
        	&\leq \frac{1}{N}\e\log \sum_{j=1}^k\int_{R(\sigma,\tau)\in I_j}\exp\bigl({\beta_n} H_{N,t}(\sigma,\tau)\bigr)dm_N(\sigma,\tau)\\
        	&\leq \frac{1}{N}\e\log \Bigl\{k\exp \Bigl(N\max_{1\leq j\leq k}\frac{1}{N}\log \int_{R(\sigma,\tau)\in I_j}\exp\bigl({\beta_n} H_{N,t}(\sigma,\tau)\bigr)dm_N(\sigma,\tau)\Bigr)\Bigr\}\\
        	&=\frac{1}{N}\log k+\e \max_{1\leq j\leq k}\frac{1}{N}\log \int_{R(\sigma,\tau)\in I_j}\exp\bigl({\beta_n} H_{N,t}(\sigma,\tau)\bigr)dm_N(\sigma,\tau)
        	\end{align*}
        	for all $N\geq 1.$ Using the Gaussian concentration inequality (see e.g. \cite{L}) for the coupled free energy and \eqref{lem5:proof}, the last inequality is further bounded above by
        	\begin{align*}
        		&\frac{1}{N}\log k+\max_{1\leq j\leq k}\mathcal{P}_{\beta_n}(t,v_j,x_{\beta_n},b_{\beta_n},\lambda_{\beta_n}(v_j))+\delta+kC\exp(-N/C)\\
        		&\leq \frac{1}{N}\log k+\sup_{v\in A_{\varepsilon/2}^{\kappa/2}(u)}\mathcal{P}_{\beta_n}(t,v,x_{\beta_n},b_{\beta_n},\lambda_{\beta_n}(v))+\delta+kC\exp(-N/C),
        	\end{align*} 
        	where $C$ is a positive constant depending only on $\xi_{\beta_n}$. Letting $N\rightarrow\infty$ and then $\delta\downarrow 0$ gives
        	\begin{align}\label{con1}
        	\limsup_{N\rightarrow\infty}\frac{1}{N}\e\log \int_{\mathcal{A}_{\varepsilon/2}^{\kappa/2}(u)}\exp \bigl({\beta_n} H_{N,t}(\sigma,\tau)\bigr)dm_N(\sigma,\tau)&\leq \sup_{v\in A_{\varepsilon/2}^{\kappa/2}(u)}\mathcal{P}_{\beta_n}(t,v,x_{\beta_n},b_{\beta_n},\lambda_{\beta_n}(v)).
        	\end{align}
        	
        	{\bf Step II:} From \eqref{con1} and the definition of $A_{\varepsilon/2}^{\kappa/2}(u)$, we may assume without loss of generality that there exists a sequence $$
        	(v_{\beta_n})\subset [-1+\kappa/2,1-\kappa/2]$$ such that the following two limits exist, $$v_0:=\lim_{n\rightarrow\infty}v_{\beta_n},\,\,\Lambda_0:=\lim_{n\rightarrow\infty}\Lambda(v_{\beta_{n}}).$$
        	In addition, for any $n\geq 1,$ 
        	\begin{align*}
        	\sup_{v\in A_{\varepsilon/2}^{\kappa/2}(u)}\mathcal{P}_{\beta_n}(t,v,x_{\beta_n},b_{\beta_n},\lambda_{\beta_n}(v))\leq \mathcal{P}_{\beta_n}(t,v_{\beta_n},x_{\beta_n},b_{\beta_n},\lambda_{\beta_n}(v_{\beta_n}))+\frac{1}{{\beta_n}}.
        	\end{align*}
        	By the continuity of $\mathcal{P}_{\beta_n}$ at the variables $u$ and $\lambda$, if $v_0\geq 0,$ we further choose $v_{\beta_n}>0$ for all $n\geq 1;$ if $v_0<0$, we take $v_{\beta_n}<0$ for all $n\geq 1.$
        	Denote $\iota_{n}=\mbox{sign}(v_{\beta_n})$ and $\iota_0=\mbox{sign}(v_0)$. Then $\lim_{n\rightarrow\infty}\iota_{n}=\iota_{v_0}.$ Observe that
        	\begin{align*}
        	\frac{b_{\beta_n}}{{\beta_n}}-\frac{d_{{\beta_n},v_{\beta_n}}^{x_{\beta_n}}(q)}{{\beta_n}}&\geq \frac{b_{\beta_n}}{{\beta_n}} -\frac{d_{\beta_n}^{x_{\beta_n}}(0)}{{\beta_n}}
        	\end{align*}
        	for any $0\leq q\leq |v_{\beta_n}|\leq 1-\kappa/2$ from \eqref{eq-12} and the left-hand side converges to $B-D_{v_0}(q)$ for all $0\leq q<|v_0|$. Consequently, the assumption that $\|\Lambda\|_\infty\leq (B-D(0))/2$ and the dominated convergence theorem together leads to
        	\begin{align*}
        	\lim_{n\rightarrow\infty}\frac{1}{{\beta_n}}\int_0^{|v_{\beta_n}|}\frac{\xi_{\beta_n}''(s)}{b_{\beta_n}+\iota_{n}\lambda_{\beta_n}(v_{\beta_n})-d_{{\beta_n},v_{\beta_n}}^{x_{\beta_n}}(q)}dq=\int_0^{|v_0|}\frac{\xi''(s)}{B+\iota_0 \Lambda_0-D_{v_0}(q)}dq.
        	\end{align*} 
        	Similarly, since
        	\begin{align*}
        	\frac{b_{\beta_n}}{{\beta_n}}-\frac{d_{{\beta_n}}^{x_{\beta_n}}(q)}{{\beta_n}}&\geq \frac{b_{\beta_n}}{{\beta_n}}-\frac{d_{\beta_n}^{x_{\beta_n}}(0)}{{\beta_n}}
        	\end{align*}
        	for any $q\in[0,1]$ and the left-hand side converges to $B-D(q)$ for all $q\in[0,1],$ using the assumption that $\|\Lambda\|_\infty\leq (B-D(0))/2$ and the dominated convergence theorem again yield
        	\begin{align*}
        	\lim_{n\rightarrow\infty}\frac{1}{{\beta_n}}\int_0^{|v_{\beta_n}|}\frac{\xi_{\beta_n}''(s)}{b_{\beta_n}\pm \iota_n\lambda_{\beta_n}(v_{\beta_n})-d_{{\beta_n},v_{\beta_n}}^{x_{\beta_n}}(q)}dq=\int_0^{|v_0|}\frac{\xi''(s)}{B\pm \iota_0\Lambda_0-D_{v_0}(q)}dq
        	\end{align*} 
        	and
        	\begin{align*}
        	\lim_{n\rightarrow\infty}\frac{1}{{\beta_n}}\int_{|v_{\beta_n}|}^1\frac{\xi_{\beta_n}''(s)}{b_{\beta_n}\pm \iota_n\lambda_{\beta_n}(v_{\beta_n})-d_{{\beta_n},v_{\beta_n}}^{x_{\beta_n}}(q)}dq=\int_{|v_0|}^1\frac{\xi''(s)}{B\pm \iota_0\Lambda_0-D_{v_0}(q)}dq.
        	\end{align*} 
        	In summary, from the definition of $\mathcal{P}_{\beta_n}(t,u,x,b,\lambda)$ in Proposition \ref{prop4}, we obtain
        	\begin{align*}
        	\lim_{n\rightarrow\infty}\frac{1}{{\beta_n}}\sup_{v\in A_{\varepsilon/2}^{\kappa/2}(u)}\mathcal{P}_{\beta_n}(t,v,x_{\beta_n},b_{\beta_n},\lambda_{\beta_n}(v))&=E(t,v_0,\Lambda_0).
        	\end{align*}
        	Now, since the function $E(t,\cdot,\cdot)$ is clearly continuous on $[-1+\kappa/2,1-\kappa/2]\times [-(B-D(0))/2,(B-D(0))/2]$, the right-hand side of the last equation can be controlled by
        	\begin{align*}
        	E(t,v_0,\Lambda_0)=\lim_{n\rightarrow\infty}E(t,v_{\beta_n},\Lambda(v_{\beta_n}))\leq \sup_{v\in{A}_{\varepsilon/2}^{\kappa/2}}E(t,v,\Lambda(v))\leq \sup_{v\in{A}_{\varepsilon/2}}E(t,v,\Lambda(v))
        	\end{align*}
        	and thus,
        	\begin{align}\label{con2}
        	\lim_{n\rightarrow\infty}\frac{1}{{\beta_n}}\sup_{v\in A_{\varepsilon/2}^{\kappa/2}(u)}\mathcal{P}_{\beta_n}(t,v,x_{\beta_n},b_{\beta_n},\lambda_{\beta_n}(v))&\leq \sup_{v\in{A}_{\varepsilon/2}}E(t,v,\Lambda(v)).
        	\end{align}
        	
        	{\bf Step III:} From \eqref{lem:dis_correlated_field:eq1}, \eqref{lem:dis_correlated_field:eq2}, \eqref{con1}, and \eqref{con2}, we conclude
        	\begin{align*}
        	\lim_{N\rightarrow\infty}\frac{1}{N}\e\max_{\mathcal{A}_\varepsilon(u)}H_{N,t}(\sigma,\tau)	\leq \sup_{v\in{A}_{\varepsilon/2}}E(t,v,\Lambda(v))+C\kappa^{1/2}.
        	\end{align*}
        	Since this inequality holds for all sufficiently small $\kappa$, the announced inequality follows  evidently.
        \end{proof}
        
        \begin{remark}\label{rm3}
        	\rm In view of the argument for \eqref{con1}, one immediately recognizes that it could also be shown that
        	\begin{align*}
        	\limsup_{N\rightarrow\infty}\frac{1}{N}\e\log \int_{\mathcal{A}_{\varepsilon/2}(u)}\exp \bigl({\beta_n} H_{N,t}(\sigma,\tau)\bigr)dm_N(\sigma,\tau)&\leq \sup_{v\in A_{\varepsilon/2}(u)}\mathcal{P}_{\beta_n}(t,v,x_{\beta_n},b_{\beta_n},\lambda_{\beta_n}(v)).
        	\end{align*}
        	However, it is unclear how to use this inequality to obtain the same upper bound \eqref{lem5:eq1} by adapting the Step II in the previous proof. The main difficulty here is that we do not know how to handle the limit of $$
        	\frac{d_{{\beta_n},v_{\beta_n}}^{x_{\beta_n}}(q)}{{\beta_n}}=\int_{0}^{1}\xi''(s){\beta_n} x_{\beta_n}(s)ds-\int_0^{|v_{\beta_n}|}\xi''(s){\beta_n} x_{\beta_n}(s)ds+\frac{1-t}{1+t}\int_q^{|v_{\beta_n}|}\xi''(s){\beta_n} x_{\beta_n}(s)ds
        	$$
        	if $|v_{\beta_n}|\rightarrow 1$. This technical obstacle could be overcome if we restrict  that $\lim_{n\rightarrow\infty}v_{\beta_n}\in [-1+\kappa/2,1-\kappa/2]$ as implemented in the Step II of the above proof. 
        \end{remark}

    \subsection{Control of $E(t,u,0)$}
	     
	    We study some properties of $E(t,u,0).$ Let $(L_0,\alpha_0)$ be the minimizer of \eqref{add:thm1:eq1}. Define $q_0=\min\mbox{supp}\alpha_0$ if $\mbox{supp}\alpha_0\neq \emptyset$ and $q_0=1$ if $\mbox{supp}\alpha_0=\emptyset.$ We divide our discussion into two parts: $|u|\leq q_0$ and $|u|\geq q_0.$

	    \subsubsection{Case I: $|u|\leq q_0$}
	   
	    For $0\leq t\leq 1$, define
	    $$
	    f_t(u)=L_0^2\bigl(t\xi'(u)+h^2\bigr)-u
	    $$
	    for $u\in[-q_0,q_0]$. If $|u|\leq q_0$, then the function $E(t,u,0)$ exhibits the following properties.

	    \begin{proposition}
	    	\label{lem1}
	    	If $q_0<1$, then for any $|u|\leq q_0$,
	    	\begin{align*}
	    	E(t,u,0)&=2GS,\\
	    	\partial_\lambda E(t,u,0)&=f_t(u).
	    	\end{align*}
	    	If $q_0=1$, then these two equations also hold for all $|u|<1.$
	    \end{proposition}
	    
	    	   We emphasize that the last two equations are not necessarily valid when $q_0=1$ and $|u|=1$ (see Remark \ref{rm4} below). The next proposition demonstrates a key feature of $f_t.$
	    
	    \begin{proposition}	\label{lem6}
		    	If $0<t<1,$ then $f_t=0$ has a unique solution, $u_t.$ Furthermore, if $h=0,$ then $u_t=0$ and if $h\neq 0,$ then $u_t\in(0,q_0).$ 
	    \end{proposition}

	   In what follows, we prove Propositions \ref{lem1} and \ref{lem6}. First, we establish a crucial property of $q_0.$
			
		\begin{lemma}\label{lem4}
			$q_0$ satisfies
			\begin{align}\label{lem4:eq1}
			L_0^2\bigl(\xi'({q}_0)+h^2\bigr)=q_0.
			\end{align} 
		\end{lemma}
		
			From equation \eqref{lem4:eq1}, one immediately sees that $q_0>0$ if $h\neq 0$. If $h=0,$ then $q_0$ could be either positive or equal to zero. To see this, one may consider Proposition \ref{add:prop1} with $h=0$ to obtain $q_0=1,$ while Proposition \ref{add:prop1.5} with $h=0$ illustrates that $q_0=0.$
			
		\begin{proof}[\rm \bf Proof of Lemma \ref{lem4}]
			Recall that ${\beta_n} x_{{\beta_n}}$ converges vaguely to $\alpha_0$ on the interval $[0,1)$ and that from \eqref{add:thm1:proof:eq2},
			\begin{align}\label{lem4:proof:eq1}
			\frac{h_{{\beta_n}}^2}{(d_{{\beta_n}}-d_{\beta_n}^{x_{{\beta_n}}}(0))^2}+\int_0^q\frac{\xi_{{\beta_n}}''(s)}{(b_{{\beta_n}}-d_{\beta_n}^{x_{{\beta_n}}}(s))^2}ds=q
			\end{align}
			for all $q\in\mbox{supp}\mu_{\beta_n}.$
			If $\mbox{supp}\alpha_0\neq \emptyset$, we choose $q_{\beta_n}'\in \mbox{supp}\mu_{{\beta_n}}$ such that $q_{\beta_n}'\rightarrow  {q}_0$. Then \eqref{lem4:proof:eq1}, \eqref{eq.0} and the bounded convergence theorem together leads to
			\begin{align}\label{eq.1}
			\frac{h^2+\xi'(q_0)}{(B-D(0))^2}=\frac{h^2}{(B-D(0))^2}+\int_0^{{q}_0}\frac{\xi''(s)}{(B-D(s))^2}ds={q}_0,
			\end{align}
			where the first equality used $D(s)=D(0)$ for $s\in[0,{q}_0)$. Recall that $q_{\beta_n}$ denotes the largest number in the support of $\mu_{\beta_n}.$ If $\mbox{supp}\alpha_0=\emptyset$, we use \eqref{lem4:proof:eq1} with $q=q_{\beta_n}$ and the same argument as above to get 
			\begin{align}\label{eq.2}
			\frac{h^2+\xi'(1)}{(B-D(0))^2}=\frac{h^2}{(B-D(0))^2}+\int_0^{1}\frac{\xi''(s)}{(B-D(0))^2}ds=1.
			\end{align} 
			Finally, since 
			\begin{align*}
			B-D(0)=\lim_{n\rightarrow\infty}\frac{1}{{\beta_n}}\bigl(b_{\beta_n}-d_{\beta_n}^{x_{\beta_n}}(0)\bigr)&=\lim_{n\rightarrow\infty}\frac{1}{\int_0^1{\beta_n} x_{\beta_n}(s)ds}=\frac{1}{L_0}
			\end{align*}
			by applying \eqref{add:thm1:proof:eq-1} and the fact $x_{\beta_n}(s)=0$ for $0\leq s\leq \inf\mbox{supp}\mu_{\beta_n}$, this completes our proof by substituting $B-D(0)=1/L_0$ in \eqref{eq.1} and \eqref{eq.2}.
		\end{proof}

		\begin{proof}[\rm \bf Proof of Proposition \ref{lem1}] We first prepare two simple facts. First, note that from  \eqref{prop3:proof:eq0},
			\begin{align*}
			F({\beta_n})
			&=\frac{1}{2}\Bigl(\frac{h_{\beta_n}^2}{b_{\beta_n}-d_{\beta_n}^{x_{\beta_n}}(0)}+\int_0^1\frac{\xi_{\beta_n}''(q)}{b_{\beta_n}-d_{\beta_n}^{x_{\beta_n}}(q)}dq+b_{\beta_n}-1-\log b_{\beta_n}-\int_0^1q\xi_{\beta_n}''(q)x_{\beta_n}(q)dq\Bigr).
			\end{align*}
			From \eqref{eq.0} and the limits obtained in the proof of Lemma \ref{prop3}, the bounded convergence theorem yields
				\begin{align}
				\begin{split}\label{prop3:eq3}
				GS&=\frac{1}{2}\Bigl(\frac{h^2}{B-D(0)}+\int_0^1\frac{\xi''(q)}{B-D(q)}dq+\xi''(1)\delta_0+\delta_0^{-1}-V_1\Bigr),
				\end{split}
				\end{align}
				where $V_1$ is defined in \eqref{more:eq9}.
			Next, from Lemma \ref{prop3} and Proposition \ref{lem2}, $D_u(q)\leq D(0)<B$ for all $q\in[0,|u|].$ A straightforward differentiation leads to
			\begin{align}\label{lem1:proof:eq1}
			\partial_\lambda T(t,u,0)&=\frac{1+t}{2}\int_0^{|u|}\frac{\iota\xi''(q)}{(B-D (q))^2}dq-\frac{1-t}{2}\int_0^{|u|}\frac{\iota\xi''(q)}{(B-D _u(q))^2}dq-u.
			\end{align}
			
			Now assume that $q_0<1.$
			For $|u|\leq q_0,$ since $\alpha_0=0$ on $[0,q_0),$ we have $D_u(q)=V_0=D (q)$ for all $q\in [0,|u|].$  Consequently, from \eqref{prop3:eq3}, $E(t,u,0)=2GS$  and from \eqref{lem1:proof:eq1},
			\begin{align*}
			\partial_\lambda E(t,u,0)&=\frac{1+t}{2}\int_0^{|u|}\frac{\iota\xi''(q)dq}{(B-D (q))^2}-\frac{1-t}{2}\int_0^{|u|}\frac{\iota\xi''(q)dq}{(B-D (q)^2)}+\frac{h^2}{(B-D (0))^2}-u\\
			&=\frac{\iota t\int_0^{|u|}\xi''(q)dq+h^2}{(B-D (0))^2}-u\\
			&=L_0^2\bigl(\iota t\xi'(|u|)h^2\bigr)-u\\
			&=f_t(u).
			\end{align*}
			Here the last equality used the fact that $\xi'$ is an odd function. If $q_0=1,$ these remain true for all $|u|<1$ by the same argument.
		\end{proof}
		
		\begin{remark}
			\label{rm4}\rm Assume that $|u|=1$, $q_0=1,$ and $t\in(0,1).$ One can check from \eqref{more:eq3} that $$
			D_u(q)=V_0(1-t)/(1+t)\neq V_0=D(q)
			$$ for all $0\leq q<1.$ Thus, by the definition of $E(t,u,0)$ and \eqref{lem1:proof:eq1}, we have $E(t,u,0)\neq 2GS$ and $\partial_\lambda E(t,u,0)\neq f_t(u).$
		\end{remark}
		
			\begin{proof}[\bf Proof of Proposition \ref{lem6}]
				Note that $\xi$ is even. Since $\xi'''$ is an odd function, $f_t$ is convex on $[0,q_0]$ and concave on $[-q_0,0].$ Assume that $h\neq 0.$ Since $f_t(0)>0$ and $f_t(q_0)<0$ by Lemma \ref{lem4}. From the intermediate value theorem and the convexity of $f_t$ on $[0,q_0]$, there exists a unique $u_t$ on $(0,q_0)$ such that $f_t(u_t)=0.$ In addition, since
				\begin{align*}
				f_t(-q_0)=-L_0^2\bigl(t\xi'(q_0)-h^2\bigr)+q_0>-L_0^2\bigl(t\xi'(q_0)+h^2\bigr)+q_0=-f_t(q_0)>0,
				\end{align*}
				the concavity of $f_t$ on $[-q_0,0]$ and $f_t(0)>0$ imply that $f_t=0$ has no solution on $[-q_0,0].$ So $f_t=0$ has only one solution $u_t$ when $h\neq 0$ and $u_t$ is located in the interval $(0,q_0).$ Suppose that $h=0.$ If $q_0=0$, then clearly $u_t=0.$ If $q_0\neq 0,$ then Lemma \ref{lem4} deduces
				$$
				f_t(-q_0)=-L_0^2t\xi'(q_0)+q_0> f_t(0)=0>L_0^2t\xi'(q_0)-q_0=f_t(q_0).
				$$  
				Using the convexity of $f_t$ on $[0,q_0]$ and the concavity of $f_t$ on $[-q_0,0]$, we conclude that $0$ is the unique solution to $f_t=0$.
			\end{proof}
		
		\subsubsection{Case II: $|u|\geq q_0$}

		We show that $E(t,u,0)$ is uniformly strictly less than $2GS$ as long as $|u|>q_0$ if $q_0\in(0,1)$ or $|u|=1$ if $q_0=1.$
		
		\begin{proposition}
			\label{lem-1}  Let $t\in(0,1).$
			\begin{itemize}
				\item[$(i)$] If $q_0\in [0,1)$, then $
				\sup_{|u|\in [s,1]}E(t,u,0)<2GS
				$ for any $s\in (q_0,1)$. 
				\item[$(ii)$] If $q_0=1$, then 
			$
				E(t,\pm 1,0)<2GS.
			$
			\end{itemize}
		\end{proposition}
		
		The validity of Proposition \ref{lem-1} relies on the following lemma. 
		
		\begin{lemma}
			Given the assumptions and the notations of Proposition \ref{lem2}, for any $u\in[-1,1]$ and $t\in[0,1],$ we have that 
			\begin{align}
			\label{lem3:eq1}
			E(t,u,0)=2GS-\mathcal{E}(t,u),
			\end{align}
			where the error term $\mathcal{E}(t,u)$ is defined as
			\begin{align}
			\begin{split}
			\label{lem3:eq2}
			\mathcal{E}(t,u)&=\frac{1-t}{2}  \int_0^{|u|} \frac{\xi''(q)(D (q)-D _u(q))}{(B-D (q))(B-D _u(q))}dq+\left\{
			\begin{array}{ll}
			0,&\mbox{if $u\in[0,1]$},\\
			\frac{h^2(D (0)-D _u(0))}{(B-D (0))(B-D _u(0))},&\mbox{if $u\in[-1,0]$.}
			\end{array}
			\right.
			\end{split}
			\end{align}
		\end{lemma}
		
		 \begin{proof}
		 	The equation \eqref{lem3:eq1} can be obtained by letting $\lambda=0$ in the right-hand side of \eqref{lem2:eq1} and noting that
		 	\begin{align*}
		 	&T(t,u,0)\\
		 	&=\frac{1+t}{2}  \int_0^{|u|} \frac{\xi''(q)}{B-D (q)} dq+ \frac{1-t}{2}  \int_0^{|u|} \frac{\xi''(q)}{B-D _u(q)} dq+\int_{|u|}^1 \frac{\xi''(q)}{B-D (q)} dq  +\xi''(1)\delta_0+\delta_0^{-1}-V_1\\
		 	&=\int_0^{1} \frac{\xi''(q)}{B-D (q)} dq+\xi''(1)\delta_0+\delta_0^{-1}-V_1+ \frac{1-t}{2}  \int_0^{|u|} \frac{\xi''(q)(D (q)-D _u(q))}{(B-D _u(q))(B-D (q))}dq\\
		 	&=2GS-\frac{h^2}{B-D(0)}+ \frac{1-t}{2}  \int_0^{|u|} \frac{\xi''(q)(D (q)-D _u(q))}{(B-D _u(q))(B-D (q))}dq,
		 	\end{align*}
		 	where the last equality used \eqref{prop3:eq3}. 
		 \end{proof}

		\begin{proof}[\rm \bf Proof of Proposition \ref{lem-1}]
			First observe that $V_0\geq \int_0^1\xi''(s)\alpha_0(s)ds$ and $V_0>0.$ In fact, the first inequality can be  seen  directly from the definitions of $V_0$ and $\alpha_0$. As for the second, if $V_0=0$, it will force $L_0=0$, which contradicts the fact that $(L_0,\alpha_0)\in\mathcal{K}$ in Lemma \ref{add:lem0}. Now, suppose $q_0\in[0,1)$ and $s\in (q_0,1).$ If $|u|\in [s,1),$ then
			$$
			D(q)-D_u(q)=\frac{2t}{1+t}\int_{q}^{|u|}\xi''(s)\alpha_0(s)ds\geq \frac{2t}{1+t}\int_{(s
				+q_0)/2}^{s}\xi''(s)\alpha_0(s)ds
			$$  
			for all $q\in [q_0,(s+q_0)/2]$; if $|u|=1,$ then
			
			$$
			D(q)-D_u(q)=\frac{2t}{1+t}\Bigl(V_0-\int_{q_0}^q\xi''(s)\alpha_0(s)ds\Bigr)\geq \frac{2t}{1+t}\Bigl(V_0-\int_{q_0}^{(s+q_0)/2}\xi''(s)\alpha_0(s)ds\Bigr)
			$$
			for all $q\in [q_0,(s+q_0)/2].$ From the observation at the beginning of the proof and the fact that $\alpha_0>0$ on $(q_0,1)$, the lower bounds on the right-hand sides of these two inequalities are positive. They together give a uniform lower bound for $D(q)-D_u(q)$ on $[q_0,(s+q_0)/2]$ for all $|u|\in [s,1].$ Consequently, from the definition of $\mathcal{E}(t,u)$ in \eqref{lem3:eq2}, $\inf_{s\leq |u|\leq 1}\mathcal{E}(t,u)>0$ and $(i)$ follows by \eqref{lem3:eq1}.
			Similarly, if $q_0=1$, then $\alpha_0=0$ on $[0,1)$ and this implies
			$$
			D(q)-D_u(q)=\frac{2t}{1+t}V_0>0
			$$
			for $q\in[0,1)$ and $u=\pm 1$. Thus, $\mathcal{E}(t,\pm 1)>0$. This and \eqref{lem3:eq1} imply $(ii)$. 
		\end{proof}

\subsection{Proof of Theorem \ref{thm4}}

Recall the quantity $u_t$ from Proposition \ref{lem6} and the maximal coupled Hamiltonian $\max_{\mathcal{A}_\varepsilon(u_t)}H_{N,t}$ from Theorem \ref{lem5}. 
The proof of Theorem \ref{thm4} is based on the following theorem, which states that as long as $0<t<1$, the maximum energy of the coupled system is essentially attained by the spin configurations, whose overlap is concentrated around the constant $u_t.$

      \begin{theorem}\label{thm1}
      	Let $0<t<1$. For any $\varepsilon\in(0,1),$ 
      	\begin{align*}
      	\limsup_{N\rightarrow\infty}\e \max_{\mathcal{A}_{\varepsilon}(u_t)}\frac{H_{N,t}(\sigma,\tau)}{N}&<2GS.
      	\end{align*}
      \end{theorem}
      
      We now use this bound to prove Theorem \ref{thm4}.
      
      	\begin{proof}[\bf Proof of Theorem \ref{thm4}]
      		Let $t\in(0,1).$ The claimed properties of $u_t$ in Theorem \ref{thm4} follows directly from Proposition \ref{lem6}. To complete the proof, observe that
      		\begin{align*}
      		H_{N,t}(\sigma_t^*,\tau_t^*)=\max_{(\sigma,\tau)\in S_N\times S_N}H_{N,t}(\sigma,\tau)&=\max_{\sigma\in S_N}H_{N,t}^1(\sigma)+\max_{\tau\in S_N}H_{N,t}^2(\tau).
      		\end{align*}
      		This implies
      		$
      		\lim_{N\rightarrow\infty}N^{-1}\e H_{N,t}(\sigma_t^*,\tau_t^*)=2GS.
      		$
      		From Theorem \ref{thm1} and the Borell inequality \cite[Theorem 7.1]{L}, there exist two constants $\eta$ and $C>0$ independent of $N$ such that with probability at least $1-Ce^{-N/C}$, the following event holds,
      		\begin{align*}
      		\frac{H_{N,t}(\sigma_t^*,\tau_t^*)}{N}-\max_{\mathcal{A}_\varepsilon(u_t)}\frac{H_{N,t}(\sigma,\tau)}{N}\geq \eta>0.
      		\end{align*}
      		This means that $(\sigma_t^*,\tau_t^*)$ is not in $\mathcal{A}_\varepsilon(u_t)$ and thus $|R(\sigma_t^*,\tau_t^*)-u_t|\geq \varepsilon$ has probability at most $Ce^{-N/C}$ and thus $\lim_{N\rightarrow\infty}\p(|R(\sigma_t^*,\tau_t^*)-u_t|\geq\varepsilon)=0.$ 
      	\end{proof}

	\begin{proof}[\rm \bf Proof of Theorem \ref{thm1}]
		Let $0<t<1.$ Recall the function $f_t$ and the quantities $q_0,u_t$ from Proposition~\ref{lem6}. Note that $0\leq u_t<1$. We split our discussion into three cases: $q_0\in(0,1),$ $q_0=1$, and $q_0=0.$ For each case, our goal is to find a measurable function $\Lambda$ on $[-1,1]$ with $\|\Lambda\|_\infty<(B-D(0))/2$  such that for any $0<\varepsilon<1-u_t$,
		\begin{align}
		\label{thm1:proof:eq2}
		\sup_{u\in {A}_{\varepsilon/2}(u_t)}E(t,u,\Lambda(u))<2GS.
		\end{align}
		If this could be accomplished, then Theorem \ref{lem5} completes our proof.
		
		First, we consider the case $q_0\in(0,1).$	Denote $K=(B-D(0))/2.$
		From Proposition \ref{lem2}, it is clear to see that $E(t,u,\cdot)$ is twice differentiable for all $|\lambda|\leq K$ and the second derivative satisfies $|\partial_{\lambda}^2 E(t,u,\lambda)|\leq C$ for some constant $C>0$. Thus, Taylor's theorem and Proposition \ref{lem1} together yield
		\begin{align}\label{thm1:proof:eq1}
		E(t,u,\lambda)\leq 2GS+\lambda f_t(u)+\frac{C\lambda^2}{2}
		\end{align}
		for all $|\lambda|\leq K.$ Define $$\lambda(u)=-\frac{\eta f_t(u)}{C}$$ for $u\in[-q_0,q_0]$, where $\eta\in(0,1)$ is chosen to be small enough such that $\max_{u\in[-q_0,q_0]}|\lambda(u)|\leq K.$
		Consequently, plugging $\lambda(u)$ into \eqref{thm1:proof:eq1} leads to
		\begin{align}\label{thm1:proof:eq3}
		E(t,u,\lambda(u))&\leq 2GS-\frac{\eta}{C}\Bigl(1-\frac{\eta}{{2}}\Bigr)f_t(u)^2.
		\end{align}
		Let $\eta_0$ be the minimum of $|f_t(u)|$ for $u\in [-q_0,q_0]$ with $|u-u_t|\geq \varepsilon/2$. From Proposition~\ref{lem6}, $\eta_0>0$ and thus,
		\begin{align}\label{eq-2}
		\sup_{u\in [-q_0,q_0]:|u-u_t|\geq \varepsilon/2}E(t,u,\lambda(u))&<2GS.
		\end{align}
		Observe from \eqref{lem2:eq2} that $E(t,\cdot,\cdot)$ is continuous on $(-1,1)\times[-K,K].$ From \eqref{eq-2}, there exists some $\eta_2\in(0,1-q_0)$ such that
		\begin{align}
		\begin{split}\label{eq-7}
		\sup_{u\in[q_0,q_0+\eta_2]}E(t,u,\lambda(q_0))<2GS,\\
		\sup_{u\in[-q_0-\eta_2,-q_0]}E(t,u,\lambda(-q_0))< 2GS.
		\end{split}
		\end{align}
		On the other hand, from Proposition \ref{lem-1}$(i)$,
		\begin{align}
				\begin{split}\label{eq-8}
		\sup_{q_0+\eta_2\leq |u|\leq 1}E(t,u,0)<2GS.
		\end{split}
		\end{align}
		Thus, if we set
		\begin{align*}
		\Lambda(u)&=\left\{
		\begin{array}{ll}
		\lambda(u),&\mbox{if $|u|\leq q_0$},\\
		\lambda(q_0),&\mbox{if $q_0<u\leq q_0+\eta_2$},\\
		\lambda(-q_0),&\mbox{if $-q_0-\eta_2\leq u< -q_0$},\\
		0,&\mbox{if $q_0+\eta_2<|u|\leq 1$},
		\end{array}\right.
		\end{align*}
		then \eqref{thm1:proof:eq2} follows immediately from \eqref{eq-2}, \eqref{eq-7}, and \eqref{eq-8}.
		
		Next we assume $q_0=1.$ In this case, we see from Proposition \ref{lem1}, \eqref{thm1:proof:eq3} still holds for all $|u|<1.$ Since $\lim_{u\rightarrow\pm 1}f_t(u)\neq 0$, it follows that
		\begin{align*}
		\sup_{u\in (-1,1):|u-u_t|\geq \varepsilon/2}E(t,u,\lambda(u))&<2GS.
		\end{align*}
		In addition, from Proposition \ref{lem-1}$(ii)$, $$
		E(t,\pm 1,0)< 2GS.
		$$
		Combining these together and taking $\Lambda(u)=\lambda(u)$ on $(-1,1)$ and $\Lambda(\pm1)=0$ yield \eqref{thm1:proof:eq2}. Finally, if $q_0=0,$ then $u_t=0.$  From Proposition \ref{lem-1}$(i)$, $\sup_{|u|\in [\varepsilon/2,1]}E(t,u,0)<2GS$. Thus, we obtain \eqref{thm1:proof:eq2} by letting $\Lambda=0$ on $[-1,1].$ In conclusion, we have constructed a measurable function $\Lambda$ for each case such that \eqref{thm1:proof:eq2} holds. This finishes our proof.
	\end{proof}

	\section{Proof of the fluctuation of the ground state energy}
	We prove Theorems \ref{thm2} and \ref{thm3}. Note $\xi$ is an even function throughout this section.
	Recall the optimizers $\sigma_t^*$ and $\tau_t^*$ for $H_{N,t}^1$ and $H_{N,t}^2$ from \eqref{more:eq4}. Our proof is based on the result of disorder chaos and the following two identities. 
	The first gives an expression of the variance of $L_N$ in terms of the overlap of $(\sigma_t^*,\tau_t^*)$, which appeared previously in the form of the Ising mixed $p$-spin models in Chatterjee \cite{Chatt13}. The second identity generalizes the first to the covariance between functions of $L_N.$ Recall the constant $\chi$ from Theorem~\ref{thm3}. 
	\begin{lemma}
		\label{id}
		We have the following two identities:
		\begin{itemize}
			\item[$(i)$] $\mbox{Var}(L_N)=N\int_0^1\e \xi(R(\sigma_t^*,\tau_t^*))dt.$
			\item[$(ii)$] Assume $h\neq 0.$ Set $$W_N=\frac{L_N-\e L_N}{\sqrt{\chi N}}\,\,\mbox{and}\,\,	W_{N,t}=\frac{L_N^1-\e L_N^1}{\sqrt{\chi N}}$$
			for $L_N^1:=\max_{\sigma\in S_N}H_{N,t}^1(\sigma)$.
			For any absolutely continuous function $\psi$ with $\|\psi'\|_\infty\leq 2,$ we have
			\begin{align}\label{id:eq2}
			\e W_N\psi(W_N)&=\frac{1}{\chi}\int_0^1\e \bigl(\psi'(W_{N,t})\xi\bigl(R(\sigma_t^*,\tau_t^*)\bigr)\bigr)dt.
			\end{align}
	
		\end{itemize}
	\end{lemma}
	
	\begin{proof}	To prove both identities, we need a technical formula for the covariance of functions of Gaussian random vectors. Let $w,w_1,w_2$ be i.i.d. centered Gaussian vectors on $\mathbb{R}^n$ with covariance $C=(C_{j,j'}).$ For $0\leq t\leq 1,$ define
		$$
		w^1(t)=\sqrt{t}w+\sqrt{1-t}w_1\,\,\mbox{and}\,\,w^2(t)=\sqrt{t}w+\sqrt{1-t}w_2.
		$$
		Let $A,B:\mathbb{R}^n\rightarrow\mathbb{R}$ be absolutely continuous functions with 
		$$
		\e \|\triangledown A(w)\|_2^2<\infty\,\,\mbox{and}\,\,	\e \|\triangledown B(w)\|_2^2<\infty.
		$$  
		From Gaussian integration by parts, we have
		\begin{align}\label{lem:id:proof:eq1}
		\e A(w)B(w)-\e A(w)\e B(w)&=\int_0^1\sum_{j,j'=1}^nC_{j,j'} \e \bigl[\partial_jA(w^1(t))\partial_{j'}B(w^2(t))\bigr]dt.
		\end{align}
		Recall the partition function $Z_N(\beta)$ and the uniform probability measure $m_N$ on $S_N$ from \eqref{more:eq5}. Let $Z_N^1(\beta)$ and $Z_N^2(\beta)$ be the partition functions associated to $H_{N,t}^1$ and $H_{N,t}^2$, respectively. Define two Gibbs measures by
		\begin{align*}
		G_{t,\beta}^1(d\sigma)&=\frac{\exp \beta H_{N,t}^1(\sigma)m_N(d\sigma)}{Z_{t}^1(\beta)}\,\,\mbox{and}\,\,G_{t,\beta}^2(d\tau)=\frac{\exp \beta H_{N,t}^2(\tau)m_N(d\tau)}{Z_{t}^2(\beta)}.
		\end{align*}
	Denote by $\la\cdot\ra_{t,\beta}$ be the Gibbs expectation with respect to $G_{t,\beta}^1\times G_{t,\beta}^2$. Since $\log Z_N(\beta)$ is a smooth function of i.i.d. Gaussian random variables $(g_{i_1,\ldots,i_p})$, we may use \eqref{lem:id:proof:eq1} with $A=B=\log Z_N(\beta)$ to obtain 
		\begin{align}\label{lem:id:proof:eq2}
		\mbox{Var}(\log Z_N(\beta))&=N\int_0^1\e \bigl\la \xi_\beta(R(\sigma,\tau))\bigr\ra_{t,\beta}dt.
		\end{align}
		Here even though the Hamiltonian $H_{N}$ may involve infinitely many $g_{i_1,\ldots,i_p}$'s, this equation still can be verified using an approximation argument by truncating the series $\xi$ by finite $p.$ Now dividing \eqref{lem:id:proof:eq2} by $\beta^2$, we obtain
		\begin{align}\label{lem:id:proof:eq3}
		\mbox{Var}\Bigl(\frac{1}{\beta}\log Z_N(\beta)\Bigr)&=N\int_0^1\e \bigl\la \xi(R(\sigma,\tau))\bigr\ra_{t,\beta}dt.
		\end{align}
		Since $\lim_{\beta\rightarrow\infty}\beta^{-1}\log Z_N(\beta)=L_N,$ the left-hand side has the limit $\mbox{Var}(L_N)$ as $\beta\rightarrow\infty.$ To see the limit on the right-hand side, we observe that  $X_N,X_N^1,X_N^2$ are linear combinations of independent even $p$-spin interactions. If $h=0,$ then both $H_{N,t}^1$ and $H_{N,t}^2$  have exactly two optimizers $\pm\sigma_t^*$ and $\pm\tau_t^*$, respectively. In this case, the Gibbs measure $G_{t,\beta}^1\times G_{t,\beta}^2$ will converge to a uniform probability measure on four points $(\pm \sigma_t^*,\pm\tau_t^*)$ when $\beta\rightarrow\infty.$ As a result,
		\begin{align*}
		\lim_{\beta\rightarrow\infty}\bigl\la \xi(R(\sigma,\tau))\bigr\ra_{t,\beta}&=\frac{1}{4}\bigl(\xi(R(\sigma_t^*,\tau_t^*))+\xi(R(-\sigma_t^*,\tau_t^*))+\xi(R(\sigma_t^*,-\tau_t^*))+\xi(R(-\sigma_t^*,-\tau_t^*))\bigr)\\
		&=\xi \bigl(R(\sigma^*,\tau^*)\bigr),
		\end{align*}
		where the last equation used the fact that $\xi$ is even. On the other hand, if $h\neq 0,$ then $\sigma_t^*$ and $\tau_t^*$ are the unique maximizers of $H_{N,t}^1$ and $H_{N,t}^2,$ in which case, the Gibbs measure $G_{t,\beta}^1\times G_{t,\beta}^2$ will converge to a Dirac measure at $(\sigma_t^*,\tau_t^*)$ when $\beta\rightarrow\infty.$ Consequently, we again have
		\begin{align}\label{more:eq6}
		\lim_{\beta\rightarrow\infty}\bigl\la \xi(R(\sigma,\tau))\bigr\ra_{t,\beta}=\xi \bigl(R(\sigma_t^*,\tau_t^*)\bigr).
		\end{align}
		Therefore, by the bounded convergence theorem, the right-hand side of \eqref{lem:id:proof:eq3} has the limit $$
		\int_0^1\e\xi\bigl(R(\sigma_t^*,\tau_t^*)\bigr)dt.$$
		This completes the proof of item $(i)$.
		
		Item $(ii)$ can be justified in a similar way. Assume $h\neq 0.$ Note that $\chi$ is positive since $u_t>0$ for all $t\in(0,1)$ by Theorem \ref{thm4}. Set 
		$$
		W_N(\beta)=\frac{\log Z_N(\beta)-\e \log Z_N(\beta)}{\beta\chi\sqrt{N}}.
		$$
		If $\psi'$ is continuous on $\mathbb{R}$, letting $A=W_N(\beta)$ and $B=\psi(W_N(\beta))$, one can argue that
		\begin{align*}
		\e W_N(\beta)\psi(W_N(\beta))-\e W_N(\beta) \e \psi(W_N)=\frac{1}{\chi}\int_0^1 \e \psi'(W_{N,t}(\beta))\la\xi(R(\sigma,\tau)) \ra_{t,\beta}dt,
		\end{align*}
		 where $W_{N,t}(\beta)$ is defined by replacing the $H_N$ in $W_N(\beta)$ by $H_{N,t}^1.$ 
		Note that $\e W_N(\beta)=0$. Passing to limit $\beta\rightarrow\infty$, the last equation and \eqref{more:eq6} together lead to \eqref{id:eq2}, If $\psi'$ is not continuous on $\mathbb{R}$, we may adapt an approximation argument to get $(ii)$ by using mollifier since $\|\psi'\|_\infty<\infty.$		
	\end{proof}

	\begin{proof}[\bf Proof of Theorem \ref{thm2}]
	Note that the assumption $h=0$ implies $u_t=0$ for all $0<t<1$ by Proposition~\ref{lem6}. From Theorem \ref{thm4} and the dominated convergence theorem, 
	\begin{align*}
	\lim_{N\rightarrow\infty}\e\xi(R(\sigma_t^*,\tau_t^*))=0.
	\end{align*}
	Using this and Lemma \ref{id}$(i)$, the result follows immediately by using the bounded convergence theorem.
\end{proof}

Finally we establish the central limit theorem for the ground state energy when $h\neq 0$ via Stein's method \cite{stein}. Recall that if $W$ is a standard normal random variable, then it satisfies the Gaussian integration by part formula, $\e W\psi(W)=\e \psi'(W)$, for any absolutely continuous function of moderate growth. Stein's method essentially utilizes the idea that if $\e W\psi(W)\approx \e\psi'(W)$ for certain class of functions, then $W$ is approximately standard normal. Specifically, to quantify the distance between $W$ and a standard normal random variable, one usually adapts Steins lemma \cite[Page 25]{stein}, which in our setting reads
\begin{align*}
d_{TV} (W_N,g )&\leq \sup\bigl\{\bigl|\e W_N\psi(W_N)-\e\psi'(W_N)\bigr|:\|\psi'\|_\infty\leq 2\bigr\}
\end{align*}
for $g$ a standard normal random variable. Now the idea is that since $R(\sigma_t^*,\tau_t^*)\approx u_t$ for all $t\in(0,1)$ by the result on chaos in disorder, the identity \eqref{id:eq2} leads to
$$
\e W_N\psi(W_N)\approx \frac{1}{\chi}\int_0^1\e \psi'(W_{N,t})\xi(u_t)dt=\e\psi'(W_N),
$$
which means that the total variance distance between $W_N$ and $g$ is asymptotically zero. This approach was originally used in \cite{Chatt091} to prove second order Poincar\'e inequalities and was recently applied to establish the central limit theorem for the free energy in the Ising mixed even $p$-spin models \cite{cpp}.

\begin{proof}[\bf Proof of Theorem \ref{thm3}]
    Assume $h\neq 0.$ Note that $u_t>0$ for all $t\in(0,1)$ by Proposition \ref{lem6}.
    From Theorem \ref{thm4}, the bounded convergence theorem yields
	\begin{align}\label{proof:thm3:eq1}
	\lim_{N\rightarrow\infty}\e \xi (R(\sigma_t^*,\tau_t^*))=\xi(u_t)
	\end{align}
	for all $t\in(0,1).$
    Since $\e\psi'(W_{N,t})=\e\psi'(W_{N})$ for all $t$, the definition of $\chi$ in Theorem \ref{thm3} and Lemma~\ref{id}$(ii)$ lead to 
	\begin{align*}
	\e W_{N}\psi(W_{N})-\e\psi'(W_{N})&=\frac{1}{\chi}\int_0^1\! \e\psi'(W_{N,t}^1)\bigl(\xi(R(\sigma_t^*,\tau_t^*))-\xi(u_t)\bigr)\, dt
	\end{align*}
	and, therefore, for any $0<\delta<1/2,$
	\begin{align*}
	\bigl|\e W_{N}\psi(W_{N})-\e\psi'(W_{N})\bigr|&\leq \frac{\|\psi'\|_\infty}{\chi}\int_\delta^{1-\delta}\! \e\bigl|\xi(R(\sigma_t^*,\tau_t^*))-\xi(u_t)| \,dt+\frac{4\delta\xi(1)\|\psi'\|_\infty}{\chi}.
	\end{align*}
	By Stein's lemma \cite[Page 25]{stein}, we obtain that
	\begin{align*}
	d_{\mathrm{TV}}(W_{N},g)&\leq \frac{2}{\chi}\int_\delta^{1-\delta}\! \e\bigl|\xi(R(\sigma_t^*,\tau_t^*))-\xi(u_t)\bigr|\,dt+\frac{8\delta\xi(1)}{\chi}.
	\end{align*}
	Using \eqref{proof:thm3:eq1} for $t\in[\delta,1-\delta]$ and the bounded convergence theorem yields
	\begin{align*}
	\limsup_{N\rightarrow\infty} d_{\text{TV}}(W_{N},g) \leq \frac{8\delta\xi(1)}{\chi}.
	\end{align*}
	Letting $\delta\downarrow 0$ finishes the proof.
	\end{proof}
	
\appendix
\setcounter{secnumdepth}{0}
\section{Appendix}

	\noindent This appendix is devoted to establishing the regularity properties of the Gaussian Hamiltonian $X_N$ as well as the proofs of Lemmas \ref{lem:dis_correlated_field} and \ref{extralem} by using the Dudley entropy integral \cite[Equation $(1.5)$]{Tal14}. For a metric space $(T,\rho)$, denote the $\delta$-covering number by $\mathcal{N}(T,\rho;\delta)$, i.e., the smallest number of open balls of radius $\delta$ that is needed in order to cover $T.$ Denote by $\|\cdot\|_2$ the Euclidean distance on $\mathbb{R}^N.$
	
	\begin{lemma}\label{lem0}
		There exists a constant $C>0$ independent of $\xi$ and $N$ such that for any $N \ge 1$ and for any $0<\delta\leq 2,$ 
		\begin{align*}
		\e \max_{\sigma,\tau\in S_N:\|\sigma-\tau\|_2\leq \delta N^{1/2}}|X_N(\sigma)-X_N(\tau)|\leq C\eta \delta N
		\end{align*}
		for all $N\geq 1,$ where $\eta:= \sqrt{\xi'(1)}$.
	\end{lemma}
	
	\begin{remark}\label{rmk2}\rm
		Fix $\sigma_0\in S_N$. Since any point $\sigma\in S_N$ satisfies  $\|\sigma-\sigma_0\|_2\leq 2 N^{1/2}$, applying Lemma \ref{lem0} with $\delta=2$ implies for all $N\geq 1,$
		\begin{align*}
		\frac{1}{N}\e \max_{\sigma\in S_N}X_N(\sigma)&\leq \frac{1}{N}\e \max_{\sigma\in S_N:\|\sigma-\sigma_0\|_2\leq 2N^{1/2}}|X_N(\sigma)-X_N(\sigma_0)| \le 2C \eta =: C_\xi.
		\end{align*}	
		
	\end{remark}
	
	\begin{proof}[\bf Proof of Lemma \ref{lem0}]
		Consider the metric \begin{align*}
		d(\sigma,\tau)&={\bigl(\e|X_N(\sigma)-X_N(\tau)|^2\bigr)}^{1/2}=(2N)^{1/2}\bigl(\xi(1)-\xi(R(\sigma,\tau))\bigr)^{1/2}.
		\end{align*}
		Note that $
		d(\sigma,\tau)\leq \eta\|\sigma-\tau\|_2.$
		Since $X_N$ is centered and
		\begin{align*}
		\log \e e^{\lambda(X_N(\sigma)-X_N(\sigma'))}&=\frac{\lambda^2 d(\sigma,\sigma')^2}{2},\,\,\forall \lambda,
		\end{align*}
		the Dudley's entropy integral (see \cite[Equation $(1.5)$]{Tal14}) implies 
		\begin{align}
		\begin{split}
		&\e\max_{\sigma,\tau\in S_N:\|\sigma-\tau\|_2\leq \delta N^{1/2}}|X_N(\sigma)-X_N(\tau)|\\\notag
		&\leq \e\max_{\sigma,\tau\in S_N:d(\sigma,\tau)\leq \eta\delta  N^{1/2}}|X_N(\sigma)-X_N(\tau)|\notag
		\end{split}\\
		\begin{split}\label{dudley} 
		&\leq C'\int_0^{\eta\delta N^{1/2}}\sqrt{\log\mathcal{N}(S_N,d;u)}du
		\end{split}
		\end{align}
		for some constant $C'>0$ independent of $\xi$ and $N.$ 
		Note that $N^{-1}S_N$ is the unit $(N-1)$-sphere with respect to $\|\cdot\|_2$, for which it is well-known (see e.g. \cite[Lemma 5.2]{RV}) that 
		\begin{align}\label{lem:discrete:proof:eq1}
		\mathcal{N}(N^{-1/2}S_N,\|\cdot\|_2;u)&\leq \Bigl(\frac{4}{u}\Bigr)^N,\,\,\forall 0<u\leq 2.
		\end{align}
		Consequently,
		\begin{align*}
		\mathcal{N}(S_N,d;u)&\leq \Bigl(\frac{4\eta N^{1/2}}{u}\Bigr)^N,\,\,\forall 0<u\leq 2 \eta \sqrt{N}
		\end{align*}
		and from \eqref{dudley} and using change of variable $v=4\eta N^{1/2}u^{-1}$, the announced statement holds by
		\begin{align*}
		\e\max_{\sigma,\tau\in S_N:\|\sigma-\tau\|_2\leq \delta N^{1/2}}|X_N(\sigma)-X_N(\tau)|&\leq C'\int_0^{\eta\delta N^{1/2}}N^{1/2}\Bigl(\log\frac{4\eta N^{1/2}}{u}\Bigr)^{1/2} du\\
		&=4\eta C' N\int_{4/\delta}^\infty\frac{(\log v)^{1/2}}{v^2}dv\\
		&\leq C\eta\delta N,
		\end{align*}
		for some constant $C$ independent of $\xi$ and $N.$
	\end{proof}

	Recall the definitions of ${A}_r(u),\mathcal{A}_r(u)$ from \eqref{thm1:eq2} and ${A}_r^{r'}(u),\mathcal{A}_r^{r'}(u)$ from \eqref{thm1:eq3}. In addition, recall the coupled Hamiltonian $H_{N,t}(\sigma,\tau)$ and measure $m_N(\sigma,\tau)$ from \eqref{hamilton2}. Define $\|(x,y)\|=\max(\|x\|_2,\|y\|_2)$ for $x,y\in\mathbb{R}^N.$
	
	\begin{proof}[\bf Proof of Lemma \ref{lem:dis_correlated_field}] This lemma can be established by following closely the discretization arguments  presented in the proof of Lemma~\ref{lem:discrete}.
		Denote
		$$
		X_{N,t}(\sigma,\tau)=X_{N,t}^1(\sigma)+X_{N,t}^2(\tau)=\sqrt{t}\bigl(X_N(\sigma)+X_N(\tau)\bigr)+\sqrt{1-t}\bigl(X_N^1(\sigma)+X_N^2(\tau)\bigr).
		$$ 
		Consider the metric
		$$
		d_t\bigl((\sigma^1,\tau^1),(\sigma^2,\tau^2)\bigr)=\bigl(\e |X_{N,t}(\sigma^1,\tau^1)-X_{N,t}(\sigma^2,\tau^2)|^2\bigr)^{1/2}.
		$$
		Since
		\begin{align}
		\begin{split}
		\label{eq-13}
		| X_{N, t}(\sigma^1, \tau^1)  - X_{N, t}(\sigma^2, \tau^2)|&\leq |X_N(\sigma^1)-X_N(\sigma^2)|+|X_N(\tau^1)-X_N(\tau^2)|\\
		&+|X_N^1(\sigma^1)-X_N^1(\sigma^2)|+|X_N^2(\tau^1)-X_N^2(\tau^2)|,
		\end{split}
		\end{align}
		it is easy to see by using the Minkowski inequality that
		\begin{align*}
		d_t\bigl((\sigma^1,\tau^1),(\sigma^2,\tau^2)\bigr)&\leq  4\eta \|(\sigma^1-\sigma^2,\tau^1-\tau^2)\|,
		\end{align*}
		where $\eta=\sqrt{\xi'(1)}$. Let $0<\delta<\min\bigl(\varepsilon,\kappa\bigr)/8.$ 
		Define
		\begin{align*}
		M_{N,t,\delta}&=\max | H_{N, t}(\sigma^1, \tau^1)  - H_{N, t}(\sigma^2, \tau^2)|,
		\end{align*}
		where the supremum above is taken over all pairs $(\sigma^1, \tau^1), (\sigma^2, \tau^2)  \in S_N\times S_N$ with  $$
		\| (\sigma^1-\sigma^2, \tau^1-\tau^2)\| \le \delta N^{1/2}.$$  Applying Lemma \ref{lem0} to the right-hand side of \eqref{eq-13} gives 
		\begin{align}
		\label{lem-1:eq1}
		\e M_{N,t,\delta}  \le\eta_0  \delta N
		\end{align}
		for some $\eta_0>0$. For any $(\sigma,\tau)\in S_N\times S_N$ and $r>0,$ set $$
		B_{N,r}(\sigma,\tau)=\bigl\{(\sigma',\tau')\in S_N\times S_N:\|(\sigma'-\sigma,\tau'-\tau)\|<rN^{1/2}\bigr\}.
		$$
		Let $\mathcal{D}_{N,\delta}$ be the collection of points in $S_N\times S_N$ with smallest cardinality such that the open balls $B_{N,\delta}(\sigma,\tau)$ for $(\sigma,\tau)\in\mathcal{D}_{N,\delta}$ form a covering of $(S_N\times S_N,\|\cdot\|).$ Note that from \eqref{lem:discrete:proof:eq1}, 
		\begin{align}
		\label{eq-3}
		|\mathcal{D}_{N,\delta}|=\mathcal{N}( S_N \times S_N, \| \cdot \|; \delta N^{1/2}) \le \left(\tfrac{4}{\delta}\right)^{2N}.
		\end{align} 
			Let $B_{N,\delta}:= m_N(B_{N,\delta}(\sigma,\tau)) $ be the weight of  each ball $B_{N,\delta}(\sigma,\tau)$. Since $m_N(S_N\times S_N)=1$, the inequality \eqref{eq-3} implies that
			\begin{align}
			\label{more:eq7}
			B_{N,\delta}\geq (4^{-1} \delta)^{2N}
			\end{align}
		Let $\mathcal{D}_{N,\delta}'$ be the collection of points $(\sigma,\tau)\in \mathcal{D}_{N,\delta}$ such that $B_{N,\delta}(\sigma,\tau)\cap  \mathcal{A}_{\varepsilon}^\kappa(u)\neq \emptyset.$
		Now on the one hand, using \eqref{lem-1:eq1}, 
		\begin{align}\label{eq-4}
		\frac{1}{N}\e \max_{\mathcal{A}_{\varepsilon}^\kappa(u)}H_{N,t}(\sigma,\tau)\leq \frac{1}{N}\e\max_{\mathcal{D}_{N,\delta}'}H_{N,t}(\sigma,\tau)+\eta_0 \delta.
		\end{align}
		On the other hand, observe that since $\delta<\min(\varepsilon,\kappa)/8$ and $B_{N,\delta}(\sigma,\tau)\cap \mathcal{A}_{\varepsilon}^\kappa(u)\neq \emptyset$ for all $(\sigma,\tau)\in \mathcal{D}_{N,\delta}'$, by the triangle inequality and the identity, 
		\begin{align}
		\label{id3}
		2N\bigl(1-R(\sigma,\tau)\bigr)=\|\sigma-\tau\|_2^2,\,\,\forall\sigma,\tau\in S_N,
		\end{align}
		it follows that $$
		B_{N,\delta}(\sigma,\tau)\subset \mathcal{A}_{\varepsilon/2}^{\kappa/2}(u),\,\,\forall (\sigma,\tau)\in \mathcal{D}_{N,\delta}'.
		$$ 
	    Consequently,
		\begin{align*}
		&\frac{1}{N}\e\max_{\mathcal{D}_{N,\delta}'}H_{N,t}(\sigma,\tau)\\
		&\leq \frac{1}{N\beta}\e\log\Bigl(\sum_{(\sigma,\tau)\in \mathcal{D}_{N,\delta}'}\frac{1}{m_N(B_{N,\delta}(\sigma,\tau))}\int_{B_{N,\delta}(\sigma,\tau)}e^{\beta H_{N,t}(\sigma,\tau)}dm_N(\sigma',\tau')\Bigr)\\
		&\leq\frac{1}{N\beta}\e\log\Bigl(\sum_{(\sigma,\tau)\in \mathcal{D}_{N,\delta}'}\frac{e^{\beta M_{N,t,\delta}}}{B_{N,\delta}}\int_{B_{N,\delta}(\sigma,\tau)}e^{\beta H_{N,t}(\sigma',\sigma')}dm_N(\sigma',\tau')\Bigr)\\
		&=\frac{\e M_{N,t,\delta}}{N} -\frac{\log B_{N,\delta}}{N\beta}+\frac{1}{N\beta}\e\log \Bigl(\sum_{(\sigma,\tau)\in\mathcal{D}_{N,\delta}'}\int_{B_{N,\delta}(\sigma,\tau)}e^{\beta H_{N,t}(\sigma',\tau')}dm_N(\sigma',\tau')\Bigr)\\
		&\leq \frac{\e M_{N,t,\delta}}{N}-\frac{\log B_{N,\delta}}{N\beta}+\frac{\log|\mathcal{D}_{N,\delta}'|}{N\beta}+\frac{1}{N\beta}\e\log \int_{\mathcal{A}_{\varepsilon/2}^{\kappa/2}(u)}e^{\beta H_{N,t}(\sigma',\tau')}dm_N(\sigma',\tau').
		\end{align*}
		From \eqref{lem-1:eq1}, \eqref{eq-3}, \eqref{more:eq7}, and \eqref{eq-4}, letting $N\rightarrow\infty$, $\beta\rightarrow \infty$ and then $\delta\rightarrow 0,$ the inequality \eqref{lem:dis_correlated_field:eq1} follows.
	\end{proof}
	
	\begin{proof}[\bf Proof of Lemma \ref{extralem}]
		We first claim that for any $(\sigma,\tau)\in\mathcal{A}_\varepsilon(u)$, there exists some $(\sigma',\tau')\in \mathcal{A}_{\varepsilon}^\kappa(u)$ such that
		\begin{align}
		\label{cond:eq1}
		\|(\sigma-\sigma',\tau-\tau')\|\leq (8\kappa N)^{1/2}.
		\end{align}
		Obviously, this inequality holds if $(\sigma,\tau)\in \mathcal{A}_\varepsilon^\kappa(u)$. So, we assume that $(\sigma,\tau)\notin \mathcal{A}_\varepsilon^\kappa(u)$. Then 
		$$
		\mbox{either $R(\sigma,\tau)>1-\kappa$ or $R(\sigma,\tau)<-1+\kappa,$}
		$$
		or equivalently, from the identity \eqref{id3},
		$$
		\mbox{either $\|\sigma-\tau\|_2<(2\kappa N)^{1/2}$ or $\|\sigma+\tau\|_2<(2\kappa N)^{1/2}.$}
		$$
		In the first case, we choose $(\sigma',\tau')$ by letting $\sigma'=\sigma$ and taking some $\tau'\in S_N$ that satisfies $\|\sigma-\tau'\|_2=(2\kappa N)^{1/2}.$ Consequently, the inequality \eqref{cond:eq1} holds since
		\begin{align*}
		\|(\sigma-\sigma',\tau-\tau')\|&=\|\tau-\tau'\|_2\\
		&\leq \|\tau-\sigma\|_2+\|\sigma-\tau'\|_2\\
		&\leq 2(2\kappa N)^{1/2}\\
		&=(8\kappa N)^{1/2}.
		\end{align*}
		In addition,  from the identity \eqref{id3} again,
		\begin{align*}
		R(\sigma',\tau')&=1-\kappa
		\end{align*}
		and from the conditions \eqref{cond},
		\begin{align*}
		R(\sigma',\tau')&= 1-\kappa\geq 1-(1-u-\varepsilon)=u+\varepsilon.
		\end{align*}
		We conclude that $(\sigma',\tau')\in\mathcal{A}_\varepsilon^\kappa(u)$. 
		Similarly, if we are in the second case, we choose $(\sigma',\tau')$ with $\sigma'=\sigma$ and some $\tau'\in S_N$ satisfying $\|\sigma+\tau'\|_2=(2\kappa N)^{1/2}.$ With this choice, the inequality \eqref{cond:eq1} follows since
		\begin{align*}
		\|(\sigma-\sigma',\tau-\tau')\|&=\|\tau-\tau'\|_2\\
		&\leq \|\tau+\sigma\|_2+\|\sigma+\tau'\|_2\\
		&\leq 2(2\kappa N)^{1/2}\\
		&=(8\kappa N)^{1/2}.
		\end{align*}
		It is also clear that $(\sigma',\tau')\in\mathcal{A}_\varepsilon^\kappa(u)$ since
		\begin{align*}
		R(\sigma',\tau')&=-1 + \kappa
		\end{align*}
		and from the conditions \eqref{cond},
		\begin{align*}
		R(\sigma',\tau')&=-1 + \kappa \leq (1+u-\varepsilon)-1=u-\varepsilon.
		\end{align*}
		Therefore, these two cases finish the proof of our claim and consequently, from \eqref{cond:eq1},
		\begin{align*}
		\frac{1}{N}\e \max_{\mathcal{A}_\varepsilon(u)}H_{N,t}(\sigma,\tau)&\leq \frac{1}{N}\e \max_{\mathcal{A}_\varepsilon^\kappa(u)}H_{N,t}(\sigma,\tau)+\frac{1}{N}\e \max_{\|(\sigma-\sigma',\tau-\tau')\|\leq (8\kappa N)^{1/2}}\bigl(H_{N,t}(\sigma,\tau)-H_{N,t}(\sigma',\tau')\bigr).
		\end{align*} 
		From Lemma \ref{lem0} and \eqref{eq-13}, the claim \eqref{lem:dis_correlated_field:eq2} follows since the second term on the right-hand side is bounded above by
		$
		4\eta(8\kappa)^{1/2}C+2(8\kappa)^{1/2}.
		$
	\end{proof}

\end{document}